\documentclass[twoside,a4paper,reqno,11pt]{amsart} 
\usepackage{amsfonts, amsbsy, amsmath, amssymb, latexsym}
\usepackage{mathrsfs,array,enumitem}
\usepackage[top=30mm,right=30mm,bottom=30mm,left=30mm]{geometry}
\usepackage{stmaryrd}
\usepackage{bm}

\usepackage{hyperref}

\renewcommand{\a}{\alpha}
\renewcommand{\b}{\beta}

\newcommand{\normeq}{\trianglelefteqslant}

\newcommand{\F}{\mathbb{F}_{q}}

\newcommand{\e}{\epsilon}
 
\renewcommand{\l}{\lambda} \renewcommand{\O}{\Omega}

 \newcommand{\s}{\sigma}
 
 \newcommand{\C}{\mathcal{C}}

\newcommand{\la}{\langle}
\newcommand{\ra}{\rangle}

\newcommand{\leqs}{\leqslant}
\newcommand{\geqs}{\geqslant}
 
\newcommand{\what}{\widehat} 
 \newcommand{\vs}{\vspace{3mm}}

\newcommand{\frat}{{\rm Frat}}
\makeatletter
\newcommand{\imod}[1]{\allowbreak\mkern4mu({\operator@font mod}\,\,#1)}
\makeatother

\newtheorem{theorem}{Theorem} 
\newtheorem*{conj*}{Conjecture}

\newtheorem{corol}[theorem]{Corollary}

\newtheorem{thm}{Theorem}[section] 
\newtheorem{prop}[thm]{Proposition} 
\newtheorem{lem}[thm]{Lemma}
\newtheorem{cor}[thm]{Corollary}

\theoremstyle{definition}
\newtheorem{rem}[thm]{Remark}
\newtheorem{remk}{Remark}

\newtheorem*{def-non}{Definition}

\begin{document}

\author{Timothy C. Burness}
\address{T.C. Burness, School of Mathematics, University of Bristol, Bristol BS8 1UG, UK}
\email{t.burness@bristol.ac.uk}

\author{Martino Garonzi}
\address{M. Garonzi, Departamento de Matem\'{a}tica, Universidade de Bras\'{i}lia, Campus Universit\'{a}rio Darcy Ribeiro, Bras\'{i}lia-DF, 70910-900, Brazil}
\email{mgaronzi@gmail.com}

 \address{R.M. Guralnick, Department of Mathematics, University of Southern California, Los Angeles CA 90089-2532, USA}
 \email{guralnic@usc.edu}
 
\author{Andrea Lucchini}
\address{A. Lucchini, Dipartimento di Matematica ``Tullio Levi-Civita”, Universit\`{a} di Padova, Via Trieste 63, 35131 Padova, Italy}
\email{lucchini@math.unipd.it}

\thanks{Garonzi acknowledges the support of the Funda\c{c}\~{a}o de Apoio \`a Pesquisa do Distrito Federal (FAPDF) and the Coordena\c{c}\~{a}o de Aperfei\c{c}oamento de Pessoal de N\'ivel Superior - Brasil (CAPES) - Finance Code 001. Guralnick was partially supported by NSF grant DMS-1600056. The authors thank two anonymous referees for helpful comments and suggestions, which have improved the clarity of the paper.}

\title{On the minimal dimension of a finite simple group} 
\dedicatory{\rm With an appendix by T.C. Burness and R.M. Guralnick}


\begin{abstract}
Let $G$ be a finite group and let $\mathcal{M}$ be a set of maximal subgroups of $G$. We say that $\mathcal{M}$ is irredundant if the intersection of the subgroups in $\mathcal{M}$ is not equal to the intersection of any proper subset. The minimal dimension of $G$, denoted ${\rm Mindim}(G)$, is the minimal size of a maximal irredundant set of maximal subgroups of $G$. This invariant was recently introduced by Garonzi and Lucchini and they computed the minimal dimension of the alternating groups. In this paper, we prove that ${\rm Mindim}(G) \leqs 3$ for all finite simple groups, which is best possible, and we compute the exact value for all non-classical simple groups. We also  introduce and study two closely related invariants denoted by $\a(G)$ and $\b(G)$. Here $\a(G)$ (respectively $\b(G)$) is the minimal size of a set of maximal subgroups (respectively, conjugate maximal subgroups) of $G$ whose intersection coincides with the Frattini subgroup of $G.$ Evidently, ${\rm Mindim}(G) \leqs \a(G) \leqs \b(G)$. For a simple group $G$ we show that $\b(G) \leqs 4$ and $\b(G) - \a(G) \leqs 1$, and both upper bounds are best possible.
\end{abstract}

\date{\today}
\maketitle

\section{Introduction}\label{s:intro}

Let $G$ be a finite group and let $\mathcal{M}$ be a set of maximal subgroups of $G$. We say that $\mathcal{M}$ is \emph{irredundant} if the intersection of the subgroups in $\mathcal{M}$ is not equal to the intersection of any proper subset of $\mathcal{M}$. Following Fernando \cite{Fern}, we define the \emph{maximal dimension} of $G$, denoted ${\rm Maxdim}(G)$, to be the maximal size of an irredundant set of maximal subgroups of $G$. This definition arises from the study of the maximum size $m(G)$ of an irredundant generating set for $G$ (that is, a generating set that does not properly contain any other generating set). Indeed, it is easy to see that $m(G) \leqs {\rm Maxdim}(G),$ and in \cite{dl} it is proved that the difference ${\rm Maxdim}(G)-m(G)$ can be arbitrarily large.

As noted in \cite{Fern}, work of Whiston \cite{Whiston} on maximal independent generating sets of the symmetric group implies that ${\rm Maxdim}(S_n)=n-1$ and ${\rm Maxdim}(A_n)=n-2$ for all $n \geqs 3$. More generally, observe that if $G$ is a nonabelian simple group then 
\[
3\leqs m(G)\leqs {\rm Maxdim}(G)
\]
since at least three involutions are needed to generate $G$. Moreover, it is worth highlighting that the maximal dimension of a simple group of Lie type $G$ can be arbitrarily large. For example, if $r$ denotes the twisted Lie rank of $G$, then a Borel subgroup is the intersection of precisely $r$ maximal parabolic subgroups and consequently ${\rm Maxdim}(G)\geqs r.$

The dual concept of \emph{minimal dimension} was introduced by Garonzi and Lucchini in \cite{GL}. We say that an irredundant set $\mathcal{M}$ of maximal subgroups is \emph{maximal irredundant} if it is not properly contained in any other irredundant set of maximal subgroups. Then the minimal dimension of $G$, denoted ${\rm Mindim}(G)$, is the minimal size of a maximal irredundant set. For example, if $G = S_3$ then $\mathcal{M} = \{\la (12) \ra, \la (13) \ra\}$ is maximal irredundant and ${\rm Mindim}(G)=2$. Note that ${\rm Mindim}(G)=1$ if and only if $G$ is cyclic of prime-power order.

The main theorem of \cite{GL} gives the exact minimal dimension of all alternating groups. More precisely, if we define
\begin{equation}\label{e:A}
\mathcal{A} = \{34,46,58, 86, 94, 106, 118, 134, 142, 146, \ldots \} 
\end{equation}
to be the set of integers of the form $2p$, where $p \ne 11$ is a prime and $2p-1$ is not a  prime power, then 
\[
{\rm Mindim}(A_n) = \left\{\begin{array}{ll}
3 & \mbox{if $n \in \{6,7,8,11,12\} \cup \mathcal{A}$} \\
2 & \mbox{otherwise}
\end{array}\right.
\]
for all $n \geqs 4$.

The proof of this result relies on earlier work \cite{BGS, James} on base sizes for primitive actions of alternating groups. To explain the connection, let $H$ be a maximal subgroup of a finite group $G$ and let $H_G = \bigcap_{g \in G}H^g$ be the core of $H$, so we can view $G/H_G$ as a primitive permutation group on the set $\O = G/H$ of cosets of $H$ in $G$. Then a subset $B$ of $\O$ is a \emph{base} for $G/H_G$ if the pointwise stabiliser of $B$ in $G/H_G$ is trivial, and we define the \emph{base size} of $G$, denoted $b(G,H)$, to be the minimal size of a base. Equivalently,
\[
b(G,H) = \min\{ |S| \,:\, S \subseteq G,\, \bigcap_{g \in S} H^g = H_G\}.
\]
Clearly, we have 
\[
{\rm Mindim}(G/H_G) \leqs b(G,H),
\]
so an upper bound on $b(G,H)$ yields an upper bound on ${\rm Mindim}(G/H_G)$. This observation leads us naturally to the following definition.

\begin{def-non}
Let $G$ be a finite group, let $\mathcal{M}$ be the set of  maximal subgroups of $G$ and let $\mathcal{M}^*$ be the set of maximal subgroups $M$ of $G$ with $M_G=\frat(G),$ the Frattini subgroup of $G$. Define
\begin{gather}\label{e:ab}
\begin{split}
\a(G) & = \min\{|\mathcal{T}| \,:\, \mathcal{T} \subseteq \mathcal{M},\, \bigcap_{H \in \mathcal{T}}H = \frat(G)\} \\
\b(G) & = \begin{cases}\min\{b(G,H) \,:\, H \in \mathcal{M^*}\}&\text{if $\mathcal{M}^*\neq \emptyset$}\\\infty&\text{otherwise}
\end{cases}
\end{split}
\end{gather}
and observe that ${\rm Mindim}(G) \leqs \a(G) \leqs \b(G)$.
\end{def-non}

By inspecting the proof of \cite[Theorem 1]{GL}, we see that 
\[
{\rm Mindim}(A_n) = \a(A_n) = \b(A_n)
\]
for all $n \geqs 4$. Our goal in this paper is to study the invariants ${\rm Mindim}(G)$, 
$\a(G)$ and $\b(G)$ for all finite simple groups. A simplified version of our main result is the following (in part (i), we define $\mathcal{A}$ as in \eqref{e:A}).

\begin{theorem}\label{t:main}
Let $G$ be a nonabelian finite simple group. 
\begin{itemize}\addtolength{\itemsep}{0.2\baselineskip}
\item[{\rm (i)}] If $G$ is an alternating, sporadic or exceptional group of Lie type, then 
\[
{\rm Mindim}(G) = \a(G) = \b(G) \leqs 3,
\]
with equality if and only if 
\[
G \in \{A_n, \, {\rm M}_{22}, \, G_2(2)' \,: \, n \in \{6,7,8,11,12\} \cup \mathcal{A}\}.
\]
\item[{\rm (ii)}] If $G$ is a classical group, then either 
\[
{\rm Mindim}(G) \leqs \a(G) \leqs \b(G) \leqs 3,
\] 
or $G = {\rm U}_{4}(2)$, ${\rm Mindim}(G)=\a(G)=3$ and $\b(G)=4$.
\end{itemize}
\end{theorem}

\begin{remk}\label{r:main1}
Let us make some comments on the statement of Theorem \ref{t:main}. 
\begin{itemize}\addtolength{\itemsep}{0.2\baselineskip}
\item[{\rm (a)}] The set $\mathcal{A}$ is infinite. To see this, let $p$ be a prime number such that $p \equiv 2 \imod{3}$ and $2p-1=q$ is a prime power. Then $q$ is a $3$-power and by combining the prime number theorem with a quantitative version of Dirichlet's theorem 
on arithmetic progressions, we conclude that $\mathcal{A}$ contains infinitely many numbers of the form $2p$ with $p \equiv 2 \imod{3}$.
Therefore, part (i) reveals that there are infinitely many finite simple groups $G$ with ${\rm Mindim}(G) = 3$.  As noted in \cite{GL}, it is not feasible to determine $\mathcal{A}$ explicitly (this is a formidably difficult problem in number theory). 
\item[{\rm (b)}] For linear groups $G = {\rm L}_{n}(q)$ we can compute all three invariants precisely. Indeed, Theorem \ref{t:lin} states that if $G \ne {\rm L}_{8}(2)$ then
\[
{\rm Mindim}(G) = \alpha(G) = \beta(G) = \left\{\begin{array}{ll} 3 & \mbox{if $G={\rm L}_{2}(7),  {\rm L}_{2}(9),  {\rm L}_{4}(2),  {\rm L}_{4}(4)$}  \\
2 & \mbox{otherwise.}
\end{array}\right.
\]
For $G = {\rm L}_{8}(2)$ we have $\a(G) = \b(G)=3$, but we have been unable to compute the exact value of ${\rm Mindim}(G)$.
\item[{\rm (c)}] Similarly, we refer the reader to Theorems \ref{t:uni}, \ref{t:symp}, \ref{t:ortodd} and \ref{t:orteven} for more detailed results for the other classical groups. It is worth noting that if $G = {\rm PSp}_{4}(2^f)'$ and $f \geqs 1$ is a $2$-power, then 
\[
{\rm Mindim}(G) \leqs \alpha(G) = \beta(G) = 3
\]
so there are infinitely many simple classical groups with $\a(G) = 3$. Let us also highlight Theorem \ref{t:ortodd}, which states that
\[
{\rm Mindim}(G) = \alpha(G) = \beta(G) = 2
\]
for all orthogonal groups $G = \O_n(q)$ with $n \geqs 7$ and $nq$ odd.
\item[{\rm (d)}] We have only identified two simple groups with $\a(G)< \b(G)$, namely $G = {\rm U}_{4}(2)$ as noted in Theorem \ref{t:main}, and $G = {\rm Sp}_{6}(4)$ with $\a(G)=2$ and $\b(G)=3$.
\end{itemize}
\end{remk}

\begin{corol}\label{c:cor1}
Let $G$ be a nonabelian finite simple group. Then the following hold:
\begin{itemize}\addtolength{\itemsep}{0.2\baselineskip}
\item[{\rm (i)}] $\a(G) \leqs 3$, with equality for infinitely many simple groups $G$.
\item[{\rm (ii)}] $\b(G) \leqs 4$, with equality if and only if $G = {\rm U}_{4}(2)$. 
\item[{\rm (iii)}] $\b(G) - \a(G) \leqs 1$, with equality if $G = {\rm U}_{4}(2)$ or ${\rm Sp}_{6}(4)$.
\item[{\rm (iv)}] $\a(G) - {\rm Mindim}(G) \leqs 1$. 
\end{itemize}
\end{corol}

We do not know if equality is possible in part (iv) of Corollary \ref{c:cor1}. However, we can show that the situation is completely different for arbitrary finite groups. Indeed, in Section \ref{examples} we construct a family of soluble groups $G$ such that the difference $\a(G) - {\rm Mindim}(G)$ is arbitrarily large.
 
\vs

It is natural to study the finite groups $G$ with ${\rm Mindim}(G)= {\rm Maxdim}(G)$, which we call \emph{minmax} groups. First observe that all nilpotent groups are  minmax. Indeed, if $G$ is nilpotent then 
\[
{\rm Mindim}(G)={\rm Maxdim}(G) = \l(|G/\frat(G)|)
\]
is the number of prime divisors of $|G/\frat(G)|$, counted with multiplicity. This is because if $H,M \leqs G$ and $M$ is maximal, then either $M$ contains $H$, or $|H:H \cap M|=|G:M|$ is a prime number, so all maximal irredundant families have the same size. It is also easy to see that there are non-nilpotent minmax groups, such as $S_3$, $A_4$ and $S_4$. In fact, one can show that any direct product of soluble minmax groups is minmax, so there are infinitely many non-nilpotent minmax groups. 

By a well-known theorem of Iwasawa \cite{Iwa}, all unrefinable chains in the subgroup lattice of a finite group $G$ have the same length if and only if $G$ is supersoluble. In our case, in place of arbitrary unrefinable chains, we restrict our attention to the unrefinable chains in the sublattice generated by the maximal subgroups of $G.$ In the context of Iwasawa's result, it is worth noting that supersoluble does not imply minmax. For instance, let $p$ be a prime such that $p-1$ is a product of at least three distinct primes and consider the affine group $G={\rm AGL}_{1}(p) = K{:}H$, where $K = C_p$ and $H = C_{p-1}$. Then $G$ is a supersoluble Frobenius group with maximal complement $H$, so ${\rm Mindim}(G)=2$. However, we have ${\rm Maxdim}(G) \geqs {\rm Maxdim}(H) \geqs 3$. 

It seems reasonable to conjecture that every minmax group is soluble, and we see that Theorem \ref{t:main} has the following corollary in support of this conjecture.

\begin{corol}\label{c:cor2}
	If $G$ is a nonabelian finite simple group, then ${\rm Mindim}(G)<{\rm Maxdim}(G).$ 
\end{corol}

Let $G$ be a nonabelian finite simple group. In order to prove Theorem \ref{t:main}, our first goal is to estimate $\b(G)$. Indeed, if there exists a maximal subgroup $H$ of $G$ with $b(G,H)=2$, then 
\[
{\rm Mindim}(G) = \a(G) = \b(G) = 2.
\]
A complete classification of the simple groups $G$ with $\b(G)=2$ remains out of reach, but we can appeal to an extensive literature on base sizes for primitive actions of almost simple groups (see \cite{B07, BGS0, BGS3, BGS, BLS, bow2} for example). Along the way, we also establish some new base size results, which may be of independent interest. For example, Lemma \ref{l:g2b} states that $b(G,H)=2$ when $G = G_2(q)$ and $H$ is a maximal rank subgroup of type ${\rm L}_{2}(q) \times {\rm L}_{2}(q)$ (the bound $b(G,H) \leqs 5$ was established in \cite{BLS}). 

\vs

The paper is structured as follows. In Section \ref{s:prel} we record some preliminary results that will be needed in the proof of Theorem \ref{t:main} and we handle the sporadic and alternating groups in Sections \ref{s:spor} and \ref{s:alt}, respectively. The exceptional groups of Lie type are studied in Section \ref{s:excep} and we state and prove our main results on classical groups in Section \ref{s:cla}. Finally, in Section \ref{s:cor} we prove Corollary \ref{c:cor2} and in Section \ref{examples} we present an example to demonstrate that there are soluble groups $G$ such that $\a(G) - {\rm Mindim}(G)$ is arbitrarily large. 

In an appendix by Burness and Guralnick, the action of the exceptional group $G_2(k)$ (either finite or algebraic) on cosets of a maximal rank subgroup of type $A_1A_1$ is studied in the even characteristic setting. Theorem \ref{t:g2main} states that this action admits a base of size $2$ when $k$ is finite, which is an essential ingredient in the proof of Lemma \ref{l:g2b}. The second main result, Theorem \ref{t:g2main2-alt}, considers the case where $k$ is an algebraically closed field and the three base measures for algebraic groups introduced in \cite{BGS2} are computed precisely. In particular, the base size for this action is $2$, but it is shown that a generic two-point stabiliser has order $2$, containing a short root element. 

\vs

Finally, let us say a few words on our notation, most of which is standard. We adopt the notation from \cite{KL} for simple groups of Lie type, so we write 
${\rm L}_{n}^{+}(q) = {\rm L}_{n}(q) = {\rm PSL}_{n}(q)$ and $E_6^{-}(q) = {}^2E_6(q)$, etc. We also use ${\rm P\O}_{n}^{\e}(q)$ to denote a simple orthogonal group, which differs from the notation in the Atlas \cite{Atlas}. A cyclic group of order $m$ is denoted by $C_m$ (or just $m$) and we write $H{:}K$ for a split extension $H$ by $K$. 
In addition, $(a,b)$ denotes the greatest common divisor of integers $a$ and $b$.

\section{Preliminaries}\label{s:prel}

In this section we record some preliminary results that will be needed in the proof of Theorem \ref{t:main}. We begin with an elementary observation, which will be used throughout the paper without further comment. Here, and for the remainder of this section, $G$ is a finite group.

\begin{lem}\label{l:easy}
Suppose $G$ has subgroups $H$ and $K$ with $|H||K| > |G|$. Then $H^g \cap K \ne 1$ for all $g \in G$. 
\end{lem}

We will also need the following generalisation.

\begin{lem}\label{l:dc}
Suppose $G$ has subgroups $H$ and $K$ and there exists a set $R \subset G$ of distinct $(H,K)$ double coset representatives such that 
\begin{itemize}\addtolength{\itemsep}{0.2\baselineskip}
\item[{\rm (i)}] $|HxK| <|H||K|$ for all $x \in R$; and 
\item[{\rm (ii)}] $\sum_{x \in R}|HxK| > |G| - |H||K|$.
\end{itemize}
Then $H^g \cap K \ne 1$ for all $g \in G$.
\end{lem}

\begin{proof}
Consider the action of $K$ on the set $\O$ of right cosets of $H$ in $G$. We may identify the $K$-orbit of $Hg$ with the double coset $HgK$, so this orbit has length
\[
|K:H^g \cap K| = \frac{|HgK|}{|H|}.
\]
By (i), the elements in $R$ correspond to distinct non-regular $K$-orbits and the inequality in (ii) implies that the union of these orbits contains more than $|\O|-|K|$ points. We conclude that $K$ does not have a regular orbit on $\O$ and the result follows.
\end{proof}

\begin{rem}
Given an appropriate group $G$, we can use {\sc Magma} \cite{magma} to implement the observation in Lemma \ref{l:dc}. Indeed, this is a straightforward extension of the double coset technique discussed in \cite[Section 2.3.3]{bow2}.
\end{rem}

Let $G$ be a finite group and let $x_1, \ldots, x_k$ be a set of representatives of the conjugacy classes in $G$ of elements of prime order. Fix a core-free subgroup $H$ of $G$. For $x \in G$, let
\[
{\rm fpr}(x,G/H) = \frac{|x^G \cap H|}{|x^G|}
\]
denote the \emph{fixed point ratio} of $x$ with respect to the standard action of $G$ on $G/H$. For a positive integer $c$ we define
\begin{equation}\label{e:qhat}
\what{Q}(G,H,c) = \sum_{i=1}^{k}|x_i^G|\,{\rm fpr}(x_i,G/H)^c.
\end{equation}
Now $\what{Q}(G,H,c)$ is an upper bound on the probability that a randomly chosen $c$-tuple of points in $G/H$ does \emph{not} form a base for $G$ (see the proof of Theorem 1.3 in \cite{LSh99}, for example). This immediately implies the following result, which is a  standard tool for bounding the base size $b(G,H)$ using fixed point ratio estimates.

\begin{lem}\label{l:prob}
If $\what{Q}(G,H,c) < 1$ then $b(G,H) \leqs c$.
\end{lem}

\vs

We are now ready to begin the proof of Theorem \ref{t:main}, which we partition into four sections according to the type of simple group we are considering. For the remainder of the paper (with the exception of Section \ref{examples}), $G$ will denote a nonabelian finite simple group and $\mathcal{M}$ is the set of maximal subgroups of $G$. We define $\a(G)$ and $\b(G)$ as in \eqref{e:ab}. 

\section{Sporadic groups}\label{s:spor}

\begin{thm}\label{t:spor}
If $G$ is a sporadic simple group, then 
\[
{\rm Mindim}(G) = \a(G) = \b(G) = \left\{\begin{array}{ll}
3 & \mbox{if $G = {\rm M}_{22}$} \\
2 & \mbox{otherwise.}
\end{array}\right.
\]
\end{thm}

\begin{proof}
From the information on base sizes presented in \cite{bow2}, we immediately deduce that 
\[
\beta(G) = \left\{\begin{array}{ll}
3 & \mbox{if $G = {\rm M}_{22}$} \\
2 & \mbox{otherwise.}
\end{array}\right.
\]

For $G = {\rm M}_{22}$, one checks that $|H||K| > |G|$ for any two non-conjugate subgroups $H,K \in \mathcal{M}$, so $\alpha(G) \geqs 3$. Finally, with the aid of {\sc Magma} \cite{magma}, it is straightforward to verify that if $A,B$ are distinct maximal subgroups of $G$, then there exists a third maximal subgroup $C \ne A,B$ such that $\{A,B,C\}$ is irredundant. This implies that ${\rm Mindim}(G) \geqs 3$ and the result follows.
\end{proof}

\section{Alternating groups}\label{s:alt} 

\begin{thm}\label{t:alt}
If $G = A_n$ with $n \geqs 5$, then 
\[
{\rm Mindim}(G) = \alpha(G) = \beta(G) = \left\{\begin{array}{ll}
3 & \mbox{if $n \in \{6,7,8,11,12\} \cup \mathcal{A}$} \\
2 & \mbox{otherwise,}
\end{array}\right.
\]
where $\mathcal{A}$ is the set of integers defined in \eqref{e:A}.
\end{thm}

\begin{proof}
This follows from the proof of \cite[Theorem 1]{GL}, but for the sake of completeness we provide a brief sketch of the main steps. 

Firstly, if $5 \leqs n \leqs 12$ then the desired result is easily checked using {\sc Magma} \cite{magma} (note that if $n \in \{6,7,8,11,12\}$, then $|H|^2>|G|$ for all $H \in \mathcal{M}$), so we may assume $n \geqs 13$. 

By the main theorem of \cite{BGS}, if there exists $H \in \mathcal{M}$ such that the action of $H$ on $\{1, \ldots, n\}$ is primitive, then $b(G,H)=2$ and thus $\beta(G)=2$. Therefore, we may assume that every maximal subgroup of $G$ is either intransitive or imprimitive. By applying a theorem of J. James \cite{James} on $b(G,H)$ for $H$ imprimitive, we can reduce to the case where $n = 2p$ and $p \geqs 7$ is a prime (see the proof of \cite[Theorem 1]{GL} for the details of this reduction). Moreover, we have $p \ne 11$ (since ${\rm M}_{22} < A_{22}$ is maximal and primitive) and $2p \ne q+1$ for a prime power $q$ (since ${\rm L}_{2}(q) < A_{q+1}$ is maximal and primitive). We have now reduced to the case where $n \in \mathcal{A}$. 

By \cite[Lemma 1]{GL}, each $H \in \mathcal{M}$ is either intransitive of the form ($S_k \times S_{n-k}) \cap G$, or imprimitive of the form $(S_p \wr S_2) \cap G$ or $(S_2 \wr S_p) \cap G$. In particular, one can check that $|H|^2 > |G|$ for all $H \in \mathcal{M}$, so $\alpha(G) \geqs 3$. Moreover, if $H = (S_2 \wr S_p) \cap G$ then $b(G,H) = 3$, as noted in \cite[Remark 1.6(ii)]{BGS}, so $\beta(G) \leqs 3$. Finally, for each pair of subgroups $A,B \in \mathcal{M}$, it is possible to construct an explicit maximal subgroup $C$ such that $\{A,B,C\}$ is irredundant (see the final step in the proof of \cite[Theorem 1]{GL}). This shows that ${\rm Mindim}(G) \geqs 3$ and the proof is complete.
\end{proof}

\section{Exceptional groups}\label{s:excep}

\begin{thm}\label{t:ex}
If $G$ is a finite simple exceptional group of Lie type, then 
\[
{\rm Mindim}(G) = \alpha(G) = \beta(G) = \left\{\begin{array}{ll} 
3 & \mbox{if $G = G_2(2)' \cong {\rm U}_{3}(3)$} \\
2 & \mbox{otherwise.}
\end{array}\right.
\]
\end{thm}

We will prove Theorem \ref{t:ex} in a sequence of lemmas. Base sizes for primitive actions of exceptional groups are studied extensively in \cite{BLS}, typically by combining fixed point ratio estimates with the upper bound in Lemma \ref{l:prob} (the parabolic actions are handled using character-theoretic methods). We will make extensive use of these results. We will also appeal to more recent results in \cite{BH2}, and we will apply work of Burness, Guralnick and Saxl \cite{BGS2} on base sizes for exceptional algebraic groups defined over an algebraically closed field. In addition, we establish some new base size results along the way, which may be of independent interest (see Lemmas \ref{l:f4b} and \ref{l:g2b}). Note that the proof of the latter result, Lemma \ref{l:g2b}, relies on Theorem \ref{t:g2main} in Appendix \ref{appendix}.

There is an extensive literature on the semisimple and unipotent conjugacy classes of simple exceptional groups (for example, \cite{Lu} is a convenient source of detailed information on semisimple classes, and similarly \cite{LieS} for unipotent classes). In particular, the sizes of these conjugacy classes are known and we will freely use this information in some of the proofs in this section.

\begin{lem}\label{l:ex1}
Theorem \ref{t:ex} holds if $G = {}^2B_2(q)$ or ${}^2G_2(q)'$.
\end{lem}

\begin{proof}
If $G = {}^2B_2(q)$ then \cite[Lemma 4.39]{BLS} gives $b(G,H)=2$ for $H = D_{2(q-1)}$. Similarly, if $G = {}^2G_2(q)$ with $q \geqs 27$ then $b(G,H) = 2$ for $H = C_{q+1}{:}C_6$ (see \cite[Lemma 4.37]{BLS}). Finally, it is easy to check that $\beta(G)=2$ when $G = {}^2G_2(3)' \cong {\rm L}_{2}(8)$. 
\end{proof}

\begin{lem}\label{l:ex2}
Theorem \ref{t:ex} holds if $G = E_6^{\e}(q)$ or $E_7(q)$.
\end{lem}

\begin{proof}
First assume $G = E_6^{\e}(q)$. Here $\mathcal{M}$ contains a maximal rank subgroup $H = {\rm L}_{3}^{\e}(q^3).3$ (see \cite[Table 5.1]{LSS}) and \cite[Lemma 6.6]{BH2} gives $b(G,H)=2$. Similarly, if $G = E_7(q)$ then \cite[Lemma 6.5]{BH2} states that $b(G,H)=2$ for $H = ({\rm L}_{2}(q^3) \times {}^3D_4(q)).3$.
\end{proof}

\begin{lem}\label{l:ex3}
Theorem \ref{t:ex} holds if $G = {}^2F_4(q)'$, ${}^3D_4(q)$ or $E_8(q)$.
\end{lem}

\begin{proof}
Suppose $G = E_8(q)$ and let $H = C_m{:}C_{30} \in \mathcal{M}$, where $m = q^8-q^7+q^5-q^4+q^3-q+1$ (see \cite[Table 5.2]{LSS}). Since $|x^G|>q^{58}$ for all nontrivial $x \in G$, we deduce that
\[
\what{Q}(G,H,2) < |H|^2q^{-58} < 1
\]
for all $q \geqs 2$ and thus $b(G,H)=2$ by Lemma \ref{l:prob}.

Next assume $G = {}^2F_4(q)'$, where $q= 2^{2m+1}$ and $m \geqs 0$. If $m=0$ then $\beta(G)=2$ (as noted in \cite[Table 11]{BLS}, we have $b(G,H)=2$ for $H = A_6.2^2$), so let us assume $m \geqs 1$ and consider $H = (C_{q+1})^2{:}{\rm GL}_{2}(3) \in \mathcal{M}$. Since $|x^G|>(q-1)q^{10}$ for all $1 \ne x \in G$, it follows that
\[
\what{Q}(G,H,2) < |H|^2(q-1)^{-1}q^{-10}<1
\]
and thus $b(G,H)=2$.

The case $G = {}^3D_4(q)$ is similar. Here we take $H = C_m{:}C_4 \in \mathcal{M}$, where $m=q^4-q^2+1$. If $q=2$ then $b(G,H)=2$ (see \cite[Table 12]{BLS}), so let us assume $q \geqs 3$. Now $|x^G| \geqs (q^8+q^4+1)(q^2-1)=a$ for all $1 \ne x \in G$, so
$\what{Q}(G,H,2) < |H|^2a^{-1}$,
which is less than $1$ for $q \geqs 4$. Finally, if $q=3$ then $H = C_{19}{:}C_4$ and thus every element in $H$ is semisimple. This implies that $|x^G| \geqs q^8(q^8+q^4+1)=b$ for all $x \in H$ of prime order and we deduce that $\what{Q}(G,H,2) < |H|^2b^{-1}<1$. 
\end{proof}

\begin{lem}\label{l:f4b}
Let $G = F_4(q)$ and let $H \in \mathcal{M}$ be a subgroup of type ${\rm L}_{3}^{\e}(q) \times {\rm L}_{3}^{\e}(q)$, where $(3,q-\e)=1$. Then $b(G,H)=2$.
\end{lem}

\begin{proof}
By \cite[Table 5.1]{LSS} we have  
\[
H = ({\rm L}_{3}^{\e}(q) \times {\rm L}_{3}^{\e}(q)).2 = B.2
\]
and it suffices to show that $\what{Q}(G,H,2)<1$, where $\what{Q}(G,H,2)$ is defined in \eqref{e:qhat} (see Lemma \ref{l:prob}). To do this, it will be convenient to write 
\[
\what{Q}(G,H,2) = \mathcal{U}+\mathcal{S},
\]
where $\mathcal{U}$ (respectively, $\mathcal{S}$) is the contribution from unipotent (respectively, semisimple) elements.

Let $\bar{G} = F_4$ be the ambient simple algebraic group over the algebraic closure of $\mathbb{F}_q$ and let $\bar{H} = A_2\tilde{A}_2$ be the connected component of the corresponding maximal closed subgroup of $\bar{G}$ (here our notation indicates that the second $A_2$ factor of $\bar{H}$ is generated by short root subgroups). Let $M$ be the natural module for $A_2$. It will be useful to consider the restriction of the Lie algebra $V = \mathcal{L}(\bar{G})$ to $\bar{H}$, which decomposes as follows 
\begin{equation}\label{e:dec}
V \downarrow \bar{H} = \mathcal{L}(\bar{H}) \oplus (M \otimes S^2(M)^*) \oplus (M^* \otimes S^2(M)) = U \oplus W \oplus W^*
\end{equation}
(see \cite[Chapter 12]{Thomas}, for example) where $S^2(M)$ denotes the symmetric-square of $M$. In addition, let us write $q=p^f$ with $p$ a prime. 

We begin by estimating $\mathcal{U}$. Let $x \in H$ be an element of order $p$. First assume $p=2$ and $x^G \cap (H \setminus B)$ is nonempty. Here we may assume that $x$ acts as a graph automorphism on the two $A_2$ factors of $\bar{H}$ and we can use the decomposition in \eqref{e:dec} to determine the Jordan form of $x$ on $V$. Indeed, we calculate that $x$ has Jordan form $[J_2^6,J_1^4]$ on $U$ and it interchanges $W$ and $W^*$, so it has Jordan form $[J_2^{24},J_1^4]$ on $V$ (here $J_i$ denotes a standard unipotent Jordan block of size $i$). By inspecting \cite[Table 4]{LawJ}, we conclude that $x$ is in the $G$-class labelled $A_1\tilde{A}_1$ in \cite[Table 22.2.4]{LieS}. Now ${\rm L}_{3}^{\e}(q)$ has a unique class of involutions (comprising root elements) and the $G$-class of each involution in $B$ is transparent. For example, if $x$ is in the $G$-class labelled $A_1$ (that is, $x$ is a long root element in $G$), then $x^G \cap H$ comprises the set of involutions in the first ${\rm L}_{3}^{\e}(q)$ factor of $B$ and thus
\[
|x^G \cap H| = \frac{|{\rm SL}_{3}^{\e}(q)|}{q^3(q-\e)} = (q+\e)(q^3-\e) < 2q^4 = a_1 = a_2
\]
and $|x^G| > q^{16} = b_1 = b_2$. The same bounds apply if $x$ is in the $\tilde{A}_1$ class. Finally, for $x$ in the class labelled $A_1\tilde{A}_1$ we get 
\[
|x^G \cap H| = \left(\frac{|{\rm SL}_{3}^{\e}(q)|}{|{\rm Sp}_{2}(q)|}\right)^2 + \left(\frac{|{\rm SL}_{3}^{\e}(q)|}{q^3(q-\e)}\right)^2 < 2q^{10} = a_3
\]
and $|x^G|>q^{28}=b_3$. 

Now assume $p$ is odd, so $x^G \cap H \subseteq B$. By inspecting \cite[Section 4.7]{Law_unip}, we deduce that the contribution to $\mathcal{U}$ from the unipotent elements in the classes $A_1$, $\tilde{A}_1$ and $A_1\tilde{A}_1$ is at most $\sum_{i=1}^{3}a_i^2b_i^{-1}$, where the $a_i$ and $b_i$ terms are defined as above. For the remaining elements $x \in G$ of order $p$, we have $|x^G|>\frac{1}{4}q^{30}=b_4$ and we note that $B$ contains precisely $q^{12} = a_4$ unipotent elements. Therefore, for any $p$, we conclude that
\[
\mathcal{U}<\sum_{i=1}^{4}a_i^2b_i^{-1}.
\]

Now let us turn to $\mathcal{S}$ and let $x \in H$ be an element of prime order $r \ne p$. Set $\bar{D} = C_{\bar{G}}(x)$. First assume $r=2$, so $\bar{D} = B_4$ or $A_1C_3$. Suppose $x^G \cap (H \setminus B)$ is nonempty. As before, at the level of algebraic groups, we may assume $x$ induces a graph automorphism on the two $A_2$ factors of $\bar{H}$ and by considering the decomposition in \eqref{e:dec}, we can determine the dimension of the $1$-eigenspace of $x$ on $V$, which coincides with the dimension of $\bar{D}$ (see \cite[Section 1.14]{Car}). Indeed, $x$ interchanges $W$ and $W^*$, and it has a $6$-dimensional $1$-eigenspace on $\mathcal{L}(\bar{H})$ (since the centraliser of a graph automorphism of $A_2$ is $3$-dimensional). It follows that $\dim C_V(x) = 24$ and thus $\bar{D} = A_1C_3$. Similarly, we can use \eqref{e:dec} to determine the $G$-class of each involution in $B$, noting that ${\rm L}_{3}^{\e}(q)$ has a unique class of involutions, with size
\[
\frac{|{\rm GL}_{3}^{\e}(q)|}{|{\rm GL}_{2}^{\e}(q)||{\rm GL}_{1}^{\e}(q)|} = q^2(q^2+\e q+1).
\]
In this way, we deduce that if $\bar{D} = B_4$, then 
\[
|x^G \cap H| =  q^2(q^2+\e q+1) < 2q^4=c_1,\;\; |x^G| > q^{16} = d_1.
\]
Similarly, if $\bar{D} = A_1C_3$, then 
\[
|x^G \cap H| = \left(\frac{|{\rm SL}_{3}^{\e}(q)|}{|{\rm SO}_{3}(q)|}\right)^2 + q^2(q^2+\e q+1)(1+q^2(q^2+\e q+1)) < 2q^{10} = c_2
\]
and $|x^G|>q^{28}=d_2$.

Finally, let us assume $r \geqs 3$, so $x^G \cap H \subseteq B$. If $\dim x^{\bar{G}} \geqs 36$, then $|x^G|>(q-1)q^{35} = d_3$ and we note that $|B|<q^{16}=c_3$. Now assume $\dim x^{\bar{G}} < 36$, in which case $\bar{D} = B_3T_1$ or $C_3T_1$, and $|x^G|>(q-1)q^{29}=d_4$. Note that $\dim C_V(x) = 22$. Write $x = x_1x_2 \in \bar{H}$, where $x_1 \in A_2$ and $x_2 \in \tilde{A}_2$. If both $x_1$ and $x_2$ are nontrivial, then using \cite[Lemma 3.7]{LSh99}, we deduce that $\dim C_W(x) = \dim C_{W^*}(x) \leqs 6$, whence $\dim C_V(x) \leqs 16$, a contradiction. Therefore, one of $x_1$ or $x_2$ is trivial and there are fewer than $2|{\rm L}_{3}^{\e}(q)|<2q^8 = c_4$ such elements in $H$.

Putting all of the above estimates together, we conclude that
\[
\what{Q}(G,H,2) < \sum_{i=1}^{4}a_i^2b_i^{-1} + \sum_{i=1}^{4}c_i^2d_i^{-1} < 1
\]
and thus $b(G,H)=2$ as claimed. 
\end{proof}

\begin{cor}\label{c:ex4}
Theorem \ref{t:ex} holds if $G = F_4(q)$.
\end{cor}

To complete the proof of Theorem \ref{t:ex} we may assume $G = G_2(q)'$. The key result is the following lemma, which states that $b(G,H)=2$ for a maximal rank subgroup $H$ of type ${\rm L}_{2}(q) \times {\rm L}_{2}(q)$. Here $\what{Q}(G,H,2)>1$ so the probabilistic approach via Lemma \ref{l:prob} is ineffective and we need to argue differently. For $q$ odd we can appeal to \cite{BGS2} where the corresponding action of the ambient simple algebraic group is studied (here it is important to note that $H$ is the centraliser of an involution). For $q$ even, this technique is not available and an entirely different approach is required (see Theorem \ref{t:g2main} in Appendix \ref{appendix}). 

\begin{lem}\label{l:g2b}
Let $G = G_2(q)$, $q \geqs 3$ and let $H \in \mathcal{M}$ be a subgroup of type ${\rm L}_{2}(q) \times {\rm L}_{2}(q)$. Then $b(G,H)=2$.
\end{lem}

\begin{proof}
The case $q$ even is handled in Appendix \ref{appendix} (see Theorem \ref{t:g2main}), so let us assume $q$ is odd, in which case $H=C_G(x)$ for an involution $x \in G$. Let $\bar{G}=G_2(k)$ be the ambient simple algebraic group, where $k$ is the algebraic closure of $\mathbb{F}_q$, and let $\s$ be a Frobenius morphism of $\bar{G}$ such that $\bar{G}_{\s}=G$. Similarly, let $\bar{H}=A_1\tilde{A}_1$ be a $\s$-stable subgroup of $\bar{G}$ such that $H = \bar{H}_{\s}$ (here the notation indicates that the second $A_1$ factor is generated by short root elements). Set $\bar{\Omega} = \bar{G}/\bar{H}$ and $\O = G/H$. Now $\s$ acts on $\bar{\Omega}$ and the natural map from $\O$ to $\bar{\Omega}_{\s}$ is an isomorphism of $H$-sets, so it suffices to show that $H$ has a regular orbit on $\bar{\Omega}_{\s}$.

By \cite[Theorem 8]{BGS2}, $\bar{H}$ has a unique regular orbit on $\bar{\Omega}$, say $\bar{\Lambda}$, and this is $\s$-stable by uniqueness. Since $\bar{H}$ is connected, the Lang-Steinberg theorem implies that $H = \bar{H}_{\s}$ acts transitively on $\bar{\Lambda}_{\s}$, whence $\bar{\Lambda}_{\s}$ is a regular $H$-orbit on $\bar{\Omega}_{\s}$ and thus $b(G,H)=2$.
\end{proof}

\begin{lem}\label{l:ex5}
Theorem \ref{t:ex} holds if $G = G_2(q)'$.
\end{lem}

\begin{proof}
By Lemma \ref{l:g2b}, we immediately deduce that 
\[
{\rm Mindim}(G) = \a(G) = \b(G)=2
\]
if $q \geqs 3$. Finally, the case $G = G_2(2)' \cong {\rm U}_{3}(3)$ can be handled using {\sc Magma}.
\end{proof}

This completes the proof of Theorem \ref{t:ex}.

\section{Classical groups}\label{s:cla}

In this section we complete the proof of Theorem \ref{t:main}. A simplified version of our main result for classical groups is the following.

\begin{thm}\label{t:cla}
Let $G$ be a finite simple classical group. Then either 
\[
{\rm Mindim}(G) \leqs \a(G) \leqs \b(G) \leqs 3,
\] 
or $G = {\rm U}_{4}(2)$, ${\rm Mindim}(G)=\a(G)=3$ and $\b(G)=4$.
\end{thm}

Let $G$ be a finite simple classical group with natural module $V$. 
The main result on the subgroup structure of $G$ is due to Aschbacher \cite{asch}, which states that each maximal subgroup of $G$ belongs to one of nine subgroup collections, denoted $\C_1, \ldots, \C_8, \mathcal{S}$. The members of the $\C_i$ collections are defined in terms of the underlying geometry of $G$. For example, they include the stabilisers of appropriate subspaces of $V$, and suitable direct sum and tensor product decompositions. The subgroups in the collection $\mathcal{S}$ are almost simple groups acting irreducibly on $V$. We refer the reader to \cite{KL} for detailed information on the structure, conjugacy and maximality of the geometric subgroups comprising the $\C_i$ collections. A complete classification of the maximal subgroups of the low-dimensional classical groups (with $\dim V \leqs 12$) is presented in \cite{BHR}. Following \cite{KL}, it will be convenient to refer to the \emph{type} of a maximal subgroup $H$ of $G$, which gives an approximate description of the group-theoretic structure of $H$.  

In studying the base sizes of primitive actions of a classical group it is natural to make a distinction between so-called \emph{subspace} and \emph{non-subspace} actions. Roughly speaking, a subspace action corresponds to the action of $G$ on an appropriate set of subspaces of the natural module (equivalently, a point stabiliser $H$ is contained in the $\C_1$ collection of reducible maximal subgroups). In this situation, the base size can be arbitrarily large. On the other hand, all non-subspace actions admit small bases. Indeed, the main theorem of \cite{B07} states that $b(G,H) \leqs 5$ for all non-subspace actions of a simple classical group and this bound is best possible. Some additional results for certain non-subspace actions are presented in \cite{BGS3, James2}, and work to extend these results is in progress (see \cite{BGS0}). The ultimate aim is to determine the base size of every primitive action of an almost simple classical group.

A key tool in the proof of Theorem \ref{t:cla} is the following result from \cite{BGS0} on the primitive actions with the property that a point stabiliser is a field extension subgroup in Aschbacher's $\C_3$ collection (see \cite[Table 4.3.A]{KL} for a description of the subgroups in $\C_3$). 

\begin{prop}\label{p:bgs}
Let $G$ be a finite simple classical group with natural module $V$ such that $\dim V \geqs 6$. Let $H \in \C_3$ be a maximal subgroup corresponding to a field extension of prime degree $k$. Then $b(G,H) \leqs 3$. More precisely, if $k \geqs 3$ then  
\[
b(G,H) = \left\{ \begin{array}{ll}
3 & \mbox{if $G = {\rm PSp}_{6}(q)$ and $H$ is of type ${\rm Sp}_{2}(q^3)$} \\
2 & \mbox{otherwise.}
\end{array}\right.
\] 
\end{prop}

\begin{proof}
This is \cite[Theorem 4.1]{BGS0}. 
\end{proof}

We will also need the following result. 

\begin{prop}\label{p:eta}
Let $G$ be a finite simple classical group with natural module $V$ such that $\dim V \geqs 6$. Let $H$ be a maximal subgroup of $G$ and suppose there is a constant $\e>0$ such that
\[
{\rm fpr}(x,G/H) < |x^G|^{-\e}
\]
for all $x \in G$ of prime order. Then $b(G,H) \leqs \left\lceil \frac{4}{3\e}\right\rceil$.
\end{prop}

\begin{proof}
Set $c = \left\lceil \frac{4}{3\e}\right\rceil$ and let $x_1, \ldots, x_k$ be representatives of the conjugacy classes in $G$ of elements of prime order. Then Lemma \ref{l:prob} implies that 
\[
\what{Q}(G,H,c) < \sum_{i=1}^{k}|x_i^G|^{1-c\e} \leqs \sum_{i=1}^{k}|x_i^G|^{-\frac{1}{3}}
\]
and this upper bound is less than $1$ by \cite[Proposition 2.2]{B07}. The result follows.
\end{proof}

\subsection{Linear groups}\label{ss:lin}

\begin{thm}\label{t:lin}
Let $G = {\rm L}_{n}(q)$, where $n \geqs 2$. If $G \ne {\rm L}_{8}(2)$ then 
\[
{\rm Mindim}(G) = \alpha(G) = \beta(G) = \left\{\begin{array}{ll} 3 & \mbox{if $G = {\rm L}_{2}(7),  {\rm L}_{2}(9),  {\rm L}_{4}(2),  {\rm L}_{4}(4)$} \\
2 & \mbox{otherwise.}
\end{array}\right.
\]
For $G = {\rm L}_8(2)$ we have ${\rm Mindim}(G) \leqs \alpha(G) = \beta(G) = 3$.
\end{thm}

\begin{proof}
First assume $n=2$. If $q$ is even, then $b(G,H) = 2$ for $H = D_{2(q-1)} \in \mathcal{M}$ (see \cite[Example 2.5]{BG}), so let us assume $q$ is odd. The cases $q \in \{5,7,9\}$ can be checked directly, and for $q \geqs 11$ we have $b(G,H) = 2$ with $H = D_{q+1} \in \mathcal{M}$ (see \cite[Lemma 7.10]{BH2}).

Next suppose $n=3$. The cases with $q<5$ can be handled directly, so let us assume $q \geqs 5$. By \cite[Table 8.3]{BHR}, $G$ contains a maximal $\C_2$-subgroup of type ${\rm GL}_{1}(q) \wr S_3$ and \cite[Theorem 1.4]{James2} gives $b(G,H)=2$. A very similar argument applies if $n=4$ or $5$ (note that the groups ${\rm L}_{4}(2) \cong A_8$ and ${\rm L}_{4}(4)$ can be handled using {\sc Magma}).

Now assume $n \geqs 6$. If $n$ is divisible by an odd prime $k$, then $G$ has a maximal $\C_3$-subgroup of type ${\rm GL}_{n/k}(q^k)$ and Proposition \ref{p:bgs} states that $b(G,H)=2$. We have now reduced to the case where $n = 2^m$ and $m \geqs 3$. If $m \geqs 4$ then a $\C_2$-subgroup $H$ of type ${\rm GL}_{4}(q) \wr S_{n/4}$ is maximal (see \cite[Table 3.5.A]{KL}) and \cite[Theorem 1.4]{James2} gives $b(G,H)=2$. Now assume $m=3$. Here we take a $\C_2$-subgroup $H$ of type ${\rm GL}_{2}(q) \wr S_{4}$, which is maximal if $q \geqs 3$ (see \cite[Table 8.44]{BHR}) and once again the result follows via  \cite[Theorem 1.4]{James2}. 

Finally, let us assume $G = {\rm L}_8(2)$ and $H$ is a maximal subgroup of $G$. If $H$ is a $\C_2$-subgroup of type ${\rm GL}_{4}(2) \wr S_2$, then one can use {\sc Magma}  to show that $b(G,H) = 3$ (more precisely, we identify sufficiently many distinct $(H,H)$ double cosets to rule out the existence of a regular $H$-orbit on $G/H$; see Lemma \ref{l:dc}). If $H$ is any other maximal subgroup, then one checks that $|H|^2 > |G|$, so  $b(G,H) \geqs 3$ and we conclude that $\beta(G)=3$. In view of Lemma \ref{l:easy}, to see that $\alpha(G) = 3$ it suffices to show that there is no $g \in G$ with $H^g \cap K = 1$, where $H$ and $K$ are of type ${\rm GL}_{4}(2) \wr S_2$ and ${\rm GL}_{4}(4)$, respectively (indeed, if $A$ and $B$ are any other non-conjugate maximal subgroups of $G$, then $|A||B|>|G|$). To do this, we use {\sc Magma} to find sufficiently many distinct $(H,K)$ double cosets to rule out the existence of a regular orbit of $K$ on $G/H$ (see Lemma \ref{l:dc}). We have not been able to determine the exact value of ${\rm Mindim}(G)$ in this case (this is difficult since  $G$ contains 7,595,740,589 maximal subgroups).
\end{proof}

\subsection{Unitary groups}\label{ss:uni}

\begin{thm}\label{t:uni}
Let $G = {\rm U}_{n}(q)$, where $n \geqs 3$. 
\begin{itemize}\addtolength{\itemsep}{0.2\baselineskip}
\item[{\rm (i)}] If $n$ is divisible by an odd prime, then
\[
{\rm Mindim}(G) = \alpha(G) = \beta(G) = \left\{\begin{array}{ll} 3 & \mbox{if $G = {\rm U}_{3}(3),  {\rm U}_{3}(5)$} \\
2 & \mbox{otherwise.}
\end{array}\right.
\]
\item[{\rm (ii)}] If $n$ is a $2$-power, then either

\vspace{1mm}

\begin{itemize}\addtolength{\itemsep}{0.2\baselineskip}
\item[{\rm (a)}] $G = {\rm U}_{4}(2)$, ${\rm Mindim}(G) = \a(G)=3$ and $\b(G)=4$, or
\item[{\rm (b)}] ${\rm Mindim}(G) \leqs \alpha(G) \leqs \beta(G) \leqs 3$.
\end{itemize}
\end{itemize}
\end{thm}

In order to prove Theorem \ref{t:uni}, we will need the following technical result.

\begin{lem}\label{l:uni}
Let $G = {\rm U}_{n}(q)$, where $n=2^m$ and $m \geqs 3$. Let $H$ be a $\C_2$-subgroup of $G$ of type ${\rm GU}_{1}(q) \wr S_n$. Then
\[
{\rm fpr}(x,G/H) < |x^G|^{-\frac{4}{9}}
\]
for all $x \in G$ of prime order.
\end{lem}

\begin{proof}
Let $x \in G$ be an element of prime order $r$ and observe that 
\[
H = ((C_{q+1})^{n-1}/Z).S_n = B.S_n,
\]
where $Z = C_{(n,q+1)}$ is the centre of ${\rm SU}_{n}(q)$ (see \cite[Proposition 4.2.9]{KL}). By the main theorem of \cite{Bur1}, we have 
\[
{\rm fpr}(x,G/H) < |x^G|^{-\frac{1}{2}+\frac{1}{n}}
\]
so we may assume $n \in \{8,16\}$. In addition, we may also assume that $r$ divides $|H|$ (otherwise ${\rm fpr}(x,G/H)=0$). The cases $(n,q) = (8,2),(8,3)$ can be handled using {\sc Magma}, so we can assume $q \geqs 4$ if $n=8$.

Suppose $(r,q+1)=1$. Then there exists a positive integer $h$ such that $hr \leqs n$ and 
\[
|x^G \cap H| \leqs \frac{n!}{h!(n-hr)!r^h}(q+1)^{h(r-1)},\;\; |x^G|> \frac{1}{4}\left(\frac{q}{q+1}\right)^{r-1}q^{a},
\]
where $a = nh(r-1)(2-hr/n)$ (see the proof of \cite[Proposition 2.5]{Bur3}). It is straightforward to check that these bounds are sufficient. 

For the remainder, we may assume $r$ divides $q+1$. In particular, $x$ is semisimple. Let $\nu(x)$ be the codimension of the largest eigenspace of $x$ on the natural module of $G$. We refer the reader to \cite[Sections 3.3 and 3.4]{Bur2} for an explanation of the bounds on $|x^G|$ presented below.

First assume $x^G \cap H \subseteq B$. If $\nu(x)=1$ then 
\[
|x^G \cap H| \leqs n,\;\; |x^G|>\frac{1}{2}\left(\frac{q}{q+1}\right)q^{2n-2}
\]
and the result follows. Similarly, if $\nu(x) \geqs 2$ then the bounds
\[
|x^G \cap H| < |B| = \frac{(q+1)^{n-1}}{(n,q+1)},\;\; |x^G|>\frac{1}{2}\left(\frac{q}{q+1}\right)q^{4n-8}
\]
are sufficient. 

To complete the proof, we may assume $r$ divides $q+1$ and $x^G \cap (H \setminus B)$ is nonempty. Notice that each primitive $r$-th root of unity occurs as an eigenvalue of $x$ with positive multiplicity. If $\nu(x)=1$ then $r=2$ and the result follows since
\[
|x^G \cap H| \leqs n+(q+1)\binom{n}{2},\;\; |x^G| > \frac{1}{2}\left(\frac{q}{q+1}\right)q^{2n-2}.
\]
Next assume $\nu(x)=2$, so $r \in \{2,3\}$. If $r=3$ then 
\[
|x^G \cap H| \leqs 2\binom{n}{2}+2\binom{n}{3}(q+1)^2,\;\; |x^G| > \frac{1}{2}\left(\frac{q}{q+1}\right)^2q^{4n-6}
\]
and similarly, 
\[
|x^G \cap H| \leqs \binom{n}{2}+3\binom{n}{4}(q+1)^2+\binom{n}{2}(n-2)(q+1),\;\; |x^G| > \frac{1}{2}\left(\frac{q}{q+1}\right)q^{4n-8}
\]
if $r=2$. In both cases, one checks that the given bounds are sufficient.

Finally, let us assume $\nu(x) \geqs 3$. Here the bounds 
\[
|x^G \cap H| < |H| = n!\left(\frac{(q+1)^{n-1}}{(n,q+1)}\right),\;\; |x^G|>\frac{1}{2}\left(\frac{q}{q+1}\right)q^{6n-18}
\]
are sufficient unless $n=16$ and $q=2,3$ (recall that we may assume $q \geqs 4$ if $n=8$). Suppose $(n,q)=(16,2)$, so $r=3$ and we note that $S_n$ contains $b=1191911840$ elements of order $3$. For $\nu(x) \geqs 4$ we have 
\[
|x^G \cap H| \leqs (b+1).3^{n-1},\;\; |x^G|>\frac{1}{2}\left(\frac{2}{3}\right)2^{8n-32}
\]
and the result follows. Similarly, if $\nu(x)=3$ then the bounds
\[
|x^G \cap H| \leqs (c+1).3^{n-1},\;\; |x^G|>\frac{1}{2}\left(\frac{2}{3}\right)2^{6n-18}
\]
are sufficient, where $c = 1120$ is the number of $3$-cycles in $S_{n}$. Finally, if $(n,q) = (16,3)$ then $r=2$ and 
\[
|x^G \cap H| \leqs (d+1).4^{n-2},\;\; |x^G|>\frac{1}{2}\left(\frac{3}{4}\right)3^{6n-18}
\]
where $d = 46206735$ is the number of involutions in $S_{n}$. Once again, it is straightforward to check that these bounds are sufficient.
\end{proof}

We are now ready to prove Theorem \ref{t:uni}.

\begin{proof}[Proof of Theorem \ref{t:uni}]
First assume $n$ is divisible by an odd prime $k$. The cases $(n,q) = (3,3)$, $(3,5)$ and $(5,2)$ can be handled directly. In the remaining cases, $G$ has a maximal subgroup $H \in \C_3$ of type ${\rm GU}_{n/k}(q^k)$ and Proposition \ref{p:bgs} gives $b(G,H)=2$.

For the remainder, we may assume $n = 2^m$ with $m \geqs 2$. Suppose $m=2$. If $q \geqs 4$ then $G$ has a maximal $\C_2$-subgroup of type ${\rm GU}_{1}(q) \wr S_4$ and $b(G,H)=2$ (see \cite[Table 2]{B07}). The cases $q \leqs 3$ can be checked directly with the aid of {\sc Magma}. In particular, for $q=2$ one checks that $\b(G)=4$, but there exist subgroups $H,K \in \mathcal{M}$ of type ${\rm GU}_{3}(q) \times {\rm GU}_{1}(q)$ and ${\rm GU}_{1}(q) \wr S_4$, respectively, such that $H \cap H^x \cap K=1$ for some $x\in G$. Therefore, $\a(G)=3$ in this case (and similarly, one checks that ${\rm Mindim}(G)=3$).

Finally, let us assume $n=2^m$ with $m \geqs 3$ and let $H$ be a $\C_2$-subgroup of type ${\rm GU}_{1}(q) \wr S_n$. By Lemma \ref{l:uni}, we have 
\[
{\rm fpr}(x,G/H) < |x^G|^{-\frac{4}{9}}
\]
for all $x \in G$ of prime order and thus $\what{Q}(G,H,3)< 1$ by Proposition \ref{p:eta}. In view of Lemma \ref{l:prob}, we conclude that $b(G,H) \leqs 3$.
\end{proof}

\subsection{Symplectic groups}\label{ss:sympl}

\begin{thm}\label{t:symp}
Let $G = {\rm PSp}_{n}(q)'$, where $n \geqs 4$.
\begin{itemize}\addtolength{\itemsep}{0.2\baselineskip}
\item[{\rm (i)}] If $n=4$ then either

\vspace{1mm}

\begin{itemize}\addtolength{\itemsep}{0.2\baselineskip}
\item[{\rm (a)}] $q=2^f$, $f \geqs 1$ a $2$-power and ${\rm Mindim}(G) \leqs \alpha(G) = \beta(G) = 3$, or 
\item[{\rm (b)}] ${\rm Mindim}(G)=\alpha(G) = \beta(G) = 2$.
\end{itemize}

\item[{\rm (ii)}] If $G = {\rm Sp}_{6}(2)$ then 
${\rm Mindim}(G) = \alpha(G) = \beta(G) = 3$.
\item[{\rm (iii)}] If $G = {\rm Sp}_{6}(4)$ then 
${\rm Mindim}(G) = \alpha(G)=2$ and $\beta(G) = 3$. 
\item[{\rm (iv)}] If $n>6$ is divisible by an odd prime, then 
${\rm Mindim}(G) = \a(G) = \b(G) = 2$.
\item[{\rm (v)}] In all other cases, 
${\rm Mindim}(G)  \leqs \alpha(G) \leqs \beta(G) \leqs 3$.
\end{itemize}
\end{thm}

\begin{rem}
Part (i) shows that there are infinitely many finite simple classical groups $G$ with $\alpha(G) = \beta(G) = 3$. It also worth noting that ${\rm Mindim}(G)=3$ for $G = {\rm Sp}_{4}(2)'$ and ${\rm Sp}_{4}(4)$, but we have not been able to compute the precise minimal dimension of the other groups arising in part (i)(a).
\end{rem}

\begin{lem}\label{l:c21}
Theorem \ref{t:symp} holds if $n=4$.
\end{lem}

\begin{proof}
Write $q=p^f$ with $p$ a prime. First assume $q$ is odd. The case $q=3$ can be handled directly, so let us assume $q \geqs 5$. Let $H$ be a $\C_2$-subgroup of type ${\rm GL}_{2}(q)$ and note that $H$ is maximal in $G$ (see \cite[Table 8.12]{BHR}). We claim that $b(G,H) = 2$.

To justify the claim, first identify $G/H$ with the set $\O$ of pairs $\{U,W\}$, where $U$ and $W$ are $2$-dimensional totally isotropic subspaces such that $V = U \oplus W$ for the natural module $V$. Fix a symplectic basis $\{e_1,e_2,f_1,f_2\}$ for $V$ and set $\a = \{U,W\} \in \O$, where $U = \la e_1, e_2 \ra$ and $W = \la f_1, f_2\ra$. Working in $L = {\rm Sp}_{4}(q)$, it suffices to show that there exists $\b =  \{U',W'\} \in \O$ with $L_{\a} \cap L_{\b} = \{\pm I_4\}$. Note that 
\[
L_{\a} = \left\{ \left(\begin{array}{cc} A & 0 \\ 0 & A^{-T} \end{array}\right),\; \left(\begin{array}{cc}  0 & -A^{-T} \\ A & 0 \end{array}\right) \,:\, A \in {\rm GL}_{2}(q) \right\}.
\]
Define
\[
U' = \la e_1, e_2+f_2 \ra,\;\; W' = \la e_1+f_2, e_2+f_1 \ra
\]
and observe that $\b = \{U',W'\} \in \O$. It is now a straightforward exercise to show that $L_{\a} \cap L_{\b} = \{\pm I_4\}$ and this justifies the claim.

Finally, let us assume $q=2^f$ is even. If $q=2$ then $G \cong A_6$ and the result follows from Theorem \ref{t:alt}, so we may assume $f>1$. If $f$ is divisible by an odd prime $k$, then we can consider a subfield subgroup $H$ of type ${\rm Sp}_{4}(q^{1/k})$ and it is straightforward to show that $b(G,H) = 2$ via Lemma \ref{l:prob} (see \cite[Table 3]{B07}).

Now assume $f = 2^m$ with $m \geqs 1$.  The cases $m \in \{1,2\}$ can be handled using {\sc Magma}, so we can assume $m \geqs 3$. Let $H$ be a subfield subgroup of type ${\rm Sp}_{4}(q^{1/2})$. By applying Lemma \ref{l:prob}, we deduce that $b(G,H) \leqs 3$. In \cite{LS}, Lawther and Saxl compute the subdegrees for the action of $G$ on $G/H$ (see \cite[Table 2]{LS}) and we immediately deduce that $H$ does not have a regular orbit, whence $b(G,H) \geqs 3$ and we conclude that $b(G,H) = 3$. By inspecting \cite[Table 8.14]{BHR}, we see that $|K|^2>|G|$ for all other maximal subgroups $K \in \mathcal{M}$, which proves that $\beta(G) = 3$. In addition, one checks that $|H||K|>|G|$ and thus $\alpha(G) = 3$. 
\end{proof}

\begin{lem}\label{l:c22}
Theorem \ref{t:symp} holds if $n \geqs 6$.
\end{lem}

\begin{proof}
First assume $n=6$. The groups with $q \leqs 4$ can be checked directly and the case $q=4$ merits special attention. Indeed, for $G = {\rm Sp}_{6}(4)$ we have $\b(G)=3$ and $\a(G)=2$ since there exists $x \in G$ such that $H \cap K^x = 1$, where $H = {\rm Sp}_{6}(2)$ is a subfield subgroup and $K$ is a $\C_3$-subgroup of type ${\rm Sp}_{2}(4^3)$. If $q \geqs 5$ then $G$ has a maximal $\C_3$-subgroup $H$ of type ${\rm Sp}_{2}(q^3)$ and Proposition \ref{p:bgs} gives $b(G,H) = 3$. 

Now assume $n>6$. If $n$ is divisible by an odd prime $k$ then $G$ has a maximal $\C_3$-subgroup of type ${\rm Sp}_{n/k}(q^k)$ with $b(G,H) = 2$. Finally, if $n = 2^m$ with $m \geqs 3$, then Proposition \ref{p:bgs} gives $b(G,H) \leqs 3$ for $H$ of type ${\rm Sp}_{n/2}(q^2)$. The result follows.
\end{proof}

\subsection{Orthogonal groups}\label{ss:ortho}

\begin{thm}\label{t:ortodd}
Let $G = \O_{n}(q)$, where $n \geqs 7$ and $nq$ is odd. Then 
\[
{\rm Mindim}(G) = \a(G) = \b(G) = 2.
\]
\end{thm}

\begin{proof}
Let $V$ be the natural module for $G$ and let $(\, ,\,)$ be the corresponding nondegenerate symmetric bilinear form on $V$. To begin with let us assume $n=4m+1$. Let 
\[
\{e_1,\dots,e_m,f_1\dots,f_m, e_1^{\ast},\dots,e_m^{\ast},f_1^{\ast},\dots,f_m^{\ast},x\}
\]
be a standard basis for $V$, where $(x,x)=1,$ $(e_i,f_i)=1,$
$(e_i^{\ast},f_i^{\ast})=1.$ We claim that the action of $G$ on the set $\O$ of $2m$-dimensional nondegenerate subspaces of $V$ of plus-type has a base of size $2$ (recall that an orthogonal $2m$-space is of \emph{plus-type} if it has a totally singular subspace of dimension $m$).

To see this, set 
\begin{align*}
U & = \langle e_1,\dots,e_m,f_1,\dots,f_m \rangle \\
W & = \langle e_1+x, f_1+e_1^{\ast}, e_2+f_1^{\ast}, f_2+e_2^{\ast}, e_3+f_2^{\ast},\dots,e_m+f_{m-1}^{\ast}, f_m+e_{m}^{\ast} \rangle
\end{align*}
and note that $U,W \in \O$. Suppose $g \in G$ stabilises $U$ and $W$, so $g$ also stabilises
\begin{align*}
U^{\perp} & = \langle x, e_1^{\ast},\dots,e_m^{\ast},f_1^{\ast},\dots,f_m^{\ast}\rangle \\
W^{\perp} & = \langle f_1-x, e_1-f_1^{\ast},f_2-e_1^{\ast},e_2-f_2^{\ast},f_3-e_2^{\ast},\dots,f_m-e_{m-1}^{\ast},e_m-f_m^{\ast},e_m^{\ast}
	\rangle.
\end{align*}	
Our goal is to prove that $g=1$. Set $Z = U^{\perp} \cap W^{\perp} = \la e_m^{\ast} \ra$.

With respect to the basis $\{e_1,f_1,\dots,e_m,f_m,x, e_1^{\ast},f_1^{\ast}, \dots,e_m^{\ast},f_m^{\ast}\}$, we claim that $g$ is represented by a block matrix of the form
\begin{equation}\label{e:block}
\begin{pmatrix}
A&0&0\\0&A&\l \\0&0& \mu
\end{pmatrix}
\end{equation}
for some $A\in O_{2m}^{+}(q)$. To see this, let us view the elements of $V$ as column vectors and set $Y=(U+W) \cap U^{\perp}=\langle x, e_1^{\ast},f_1^{\ast}, \dots,e_m^{\ast}\rangle.$ Since $g$ stabilises $U$ and $Y$, and it sends $f_m^*$ into $U^\perp,$ it follows that $g$ is represented by a block matrix of the form
\[
\begin{pmatrix}
A&0&0\\0&B&\l \\0&0& \mu
\end{pmatrix}
\]
for some $A, B\in O_{2m}^{+}(q)$. We may define an invertible linear map $\sim$ from $U$ to $Y$ setting $\tilde e_1=x,$ $\tilde f_1=e_1^*, \tilde e_2=f_1^*,\dots,\tilde e_m=f_{m-1}^*, \tilde f_m=e_m^*.$
Notice that if $u\in U$ and $y\in Y$, then  $u+y\in W$ if and only if $y=\tilde u.$ Since $g$ stabilises $W$, we have $(u+\tilde u)^g=Au+B\tilde u \in W$ for every $u\in U$ and this implies that $A=B$. This justifies the claim.

We record some useful facts:

\begin{enumerate}\addtolength{\itemsep}{0.2\baselineskip}
\item[(a)] There exist $b_0,b_1,\ldots,b_m \in \mathbb{F}_q$ such that 
\begin{align*}
(f_i)^{\alpha} & = \sum_{j=i}^m b_{j-i}f_j \mbox{ for all $i=1,\ldots,m$} \\
(e_i^{\ast})^{\alpha} & = \sum_{j=i}^m b_{j-i} e_j^{\ast} \mbox{ for all $i=0,\ldots,m$}
\end{align*}
where $e_0^{\ast}=x$. To prove this we apply descending induction on $i$. 

Suppose $i=m$. Since $g$ stabilises $Z$, there exists a nonzero scalar $a=b_0 \in \F$ such that $(e_m^{\ast})^g=a e_m^{\ast}$, whence $f_m^{g}=af_m$. Now assume $i<m$. By induction, we have $f_{i+1}^{g} = \sum_{j=i+1}^m b_{j-i-1}f_j$, so 
\begin{equation}\label{e:ast}
(e_i^{\ast}-f_{i+1})^{g}-\sum_{j=i+1}^m b_{j-i-1} (e_{j-1}^{\ast}-f_j) = (e_i^{\ast})^g -\sum_{j=i+1}^m b_{j-i-1} e_{j-1}^{\ast} \in Z
\end{equation}
and thus there exists $b_{m-i} \in \F$ such that the above vector equals $b_{m-i}e_m^{\ast}$, implying 
\[
(e_i^{\ast})^{g} = \sum_{j=i}^m b_{j-i}e_j^{\ast}.
\] 
By the block matrix form of $g$ in \eqref{e:block}, it follows that 
\[
f_i^{g} = \sum_{j=i}^m b_{j-i}f_j.
\]

\item[(b)] We have 
\begin{align*}
x^g & = ax+b_1e_1^{\ast}+\ldots+ b_{m-1}e_{m-1}^{\ast}+b_me^{\ast}_m\\
e_1^g & = ae_1+b_1f_1+\ldots+ b_{m-1}f_{m-1}+b_mf_m
\end{align*}

\item[(c)] There exist $c_1,\ldots,c_{m-1} \in \F$ such that 
\begin{align*}
e_i^{g} & = \sum_{j=i}^m b_j f_{j-i+1}+ae_i+\sum_{j=1}^{i-1}c_j f_{m-i+j+1} \\
(f_{i-1}^{\ast})^{g} & = \sum_{j=i-1}^m b_je_{j-i+1}^{\ast}+af_{i-1}^{\ast}+\sum_{j=1}^{i-1} c_je_{m-i+j+1}^{\ast}
\end{align*}
for all $i=2, \ldots, m$. The proof is very similar to the argument in item (a). 
\end{enumerate}

Using $(e_i^{g},e_i^{g})=(e_i,e_i)=0$ and $((f_i^{\ast})^{g},(f_i^{\ast})^{g})=(f_i^{\ast},f_i^{\ast})=0$ we quickly deduce that $b_j=0$ for all $1 \leqs j \leqs m$ and $c_j=0$ for all $1 \leqs j < m$. In addition, by applying the argument in \eqref{e:ast} to the vector $f_m^{\ast}-e_m$ we deduce that $\l=0$ and $\mu=a$ in \eqref{e:block}, hence $g = aI_{n}$ and thus $a=1$. 

\vs

Now assume $n=4m+3$. Here a very similar argument applies and we provide a sketch.  Fix a standard basis 
\[
\{ e_1,\dots,e_m,f_1\dots,f_m, e_1^{\ast},\dots,e_m^{\ast},f_1^{\ast},\dots,f_m^{\ast},e,f,x\}
\]
for $V$, where $(x,x)=1,$ $(e,f)=1,$ $(e_i,f_i)=1,$ $(e_i^{\ast},f_i^{\ast})=1.$ Let $\O$ be the set of $(2m+1)$-dimensional nondegenerate subspaces $X$ of $V$ with the property that $X^{\perp}$ has plus-type. Then $G$ acts primitively on $\O$ and we claim that there is a base of size $2$.

Set 
\begin{align*}
U & = \langle x, e_1,\dots,e_m,f_1,\dots,f_m\rangle \\
W & = \langle x+e_1^{\ast}, e_1+f_1^{\ast}, f_1+e_2^{\ast},\dots, e_m+f_m^{\ast}, f_m+e 
\rangle
\end{align*}
and observe that $U,W \in \O$. Suppose $g \in G$ stabilises $U$ and $W$, in which case $g$ also stabilises 
\begin{align*}
U^{\perp} & = \langle e_1^{\ast},\dots,e_m^{\ast},f_1^{\ast},\dots,f_m^{\ast},e,f\rangle \\
W^{\perp} & =\langle x-f_1^{\ast}, f_1-e_1^{\ast}, e_1-f_2^{\ast},\dots,f_m-e_m^{\ast},e_m-f,e \rangle.
\end{align*}
With respect to the basis 
\[
\{x,e_1,f_1,\dots,e_m,f_m, e_1^{\ast},f_1^{\ast}, \dots,e_m^{\ast},f_m^{\ast},e,f\},
\]
the element $g$ is represented by the same block matrix as in \eqref{e:block}, where $A \in O_{2m+1}(q)$. Set $e_{m+1}^{\ast}=e$. We have the following facts:

\begin{enumerate}\addtolength{\itemsep}{0.2\baselineskip}
\item[(a)] There exist $b_1,\ldots,b_m \in \F$ such that 
\[
(e_i^{\ast})^{g} = ae_i^{\ast}+\sum_{j=1}^{m-i+1} b_j e_{i+j}^{\ast}, \;\; f_{i-1}^{g} = af_{i-1}+\sum_{j=1}^{m-i+1} b_jf_{i+j-1}
\]
for $i=2,\ldots,m$. 
\item[(b)] There exist $b_{m+1},\ldots,b_{2m} \in \F$ such that 
\begin{align*}
e_i^{g} & = ae_i + \sum_{j=1}^{i-1} b_j e_{i-j}+b_ix+\sum_{j=i+1}^{m+i} b_j f_{j-i} \\
(f_i^{\ast})^{g} & = a f_i^{\ast}+\sum_{j=1}^{i-1} b_jf_{i-j}^{\ast}+b_ie_1^{\ast}+\sum_{j=i+1}^{m+i} b_je_{j-i+1}^{\ast}
\end{align*}
for $i=1,\ldots,m$. 
\end{enumerate}
Using $0=(e_i,e_i)=(e_i^{g},e_i^{g})$ and $0=(f_i^{\ast},f_i^{\ast})=((f_i^{\ast})^{g},(f_i^{\ast})^{g})$ we deduce that $b_j=0$ for all $j$. Finally, by applying the argument in \eqref{e:ast} to the vector $f-e_m \in W^{\perp}$, we see that $\l=0$ and $\mu=a$. Therefore $g = aI_n$ and thus $a=1$ and $g=1$.
\end{proof}

For the even-dimensional orthogonal groups, we will need the following technical result.  In the statement of the lemma, we work with a standard basis 
\begin{equation}\label{e:B}
\mathcal{B} = \{ e_1,\ldots,e_m,e_1^{\ast},\ldots,e_m^{\ast},f_1,\ldots,f_m,f_1^{\ast},\ldots,f_m^{\ast} \}
\end{equation}
for the natural module $V$ of $G=\O_{4m}^+(q)$, where $Q(e_i) = Q(e_i^*) = Q(f_i) = Q(f_i^*)=0$ and $(e_i,f_j) = (e_i^*,f_j^*)=\delta_{i,j}$ (here $Q$ is the defining quadratic form on $V$ and $(\, ,\,)$ is the corresponding symmetric bilinear form). In addition, for a vector $e = \sum_{i}a_ie_i$ we define $e^* = \sum_{i}a_ie_i^*$, and similarly if $f = \sum_ia_if_i$ then $f^* = \sum_ia_if_i^*$.

\begin{lem}\label{even}
Let $G=\O_{4m}^+(q)$, where $m \geqs 2$ and $q$ is even, and set
\[
\begin{array}{lll}
E=\langle e_1,\ldots,e_m\rangle & F=\langle f_1,\ldots,f_m\rangle &  W=\langle E,F \rangle \\
E^{\ast}=\langle e_1^{\ast},\ldots,e_m^{\ast}\rangle & F^{\ast}=\langle f_1^{\ast},\ldots,f_m^{\ast}\rangle & W^{\ast}=\langle E^{\ast},F^{\ast} \rangle
\end{array}
\]
with respect to the basis $\mathcal{B}$. Fix $A,B \in {\rm GL}_m(q)$ such that $\langle A,B \rangle = {\rm GL}_m(q)$ and $I_m+A$ is invertible. Set 
\begin{align*}
W_1 & = \langle e+(Ae)^{\ast},\ f+f^{\ast} \,:\, e \in E,\, f \in F \rangle \\
W_2 & = \langle e+(Be)^{\ast},f\, :\, e \in E,\, f \in F\rangle.
\end{align*}
Then every element of $G$ stabilising $W,W_1$ and $W_2$ has the form 
\[
T_a =\begin{pmatrix}aI_{2m} & 0 \\ 0&a^{-1}I_{2m} \end{pmatrix}
\]
with respect to the basis $\mathcal{B}$, for some nonzero $a \in \F$.
\end{lem}

\begin{proof}
First observe that $W$, $W_1$ and $W_2$ are nondegenerate $(2m)$-spaces of plus-type (note that $W_1$ is nondegenerate since $I_m+A$ is invertible). Suppose $g \in G$ stabilises $W,W_1$ and $W_2$. Then $g$ stabilises $W \cap W_2=F$ and the radical of $W+W_2=\langle E,F,E^{\ast} \rangle$, which is $E^{\ast}$. Moreover, since $g$ stabilises $W$ it also stabilises $W^{\ast} = W^{\perp}$. Writing $g$ in block form using the ordered basis $\mathcal{B}$ in \eqref{e:B} we deduce that for some $m \times m$ matrices $R,S,X_1,X_2,X_3,X_4$, with the $X_i$ invertible, we have 
\[
g=\left( \begin{array}{cccc} X_1 & 0 & 0 & 0 \\ 0 & X_2 & 0 & S \\ R & 0 & X_3 & 0 \\ 0 & 0 & 0 & X_4 \end{array} \right).
\]
Since $g$ stabilises $W_1$ we quickly deduce that $X_3=X_4$, $X_2=AX_1A^{-1}$ and $R=S=0$. Moreover, since $g$ preserves the underlying symplectic form on $V$, it follows that $X_3=X_1^{-T}$ and $X_4=X_2^{-T}$. But $X_3=X_4$ and we deduce that $X_1=X_2=AX_1A^{-1}$, in other words $X_1=X_2$ commutes with $A$. Since $g$ stabilises $W_2$ we quickly deduce that $X_2B=BX_2$, so $X_1=X_2$ also commutes with $B$. Finally, since $\langle A,B \rangle = {\rm GL}_m(q)$, it follows that $X_1=X_2=aI_m$ for some nonzero scalar $a\in \F$ and the result follows.
\end{proof}

We can now establish our main result for even-dimensional orthogonal groups.

\begin{thm}\label{t:orteven}
Let $G = {\rm P\O}_{n}^{\e}(q)$, where $n \geqs 8$ is even. 
\begin{itemize}\addtolength{\itemsep}{0.2\baselineskip}
\item[{\rm (i)}] If $n$ is divisible by an odd prime $k$ with $n/k \geqs 4$, then 
\[
{\rm Mindim}(G) = \alpha(G) = \beta(G) = 2.
\]
\item[{\rm (ii)}] If $G = \O_{8}^{+}(2)$, then 
${\rm Mindim}(G) = \alpha(G) = \beta(G) = 3$.
\item[{\rm (iii)}] In all other cases, ${\rm Mindim}(G) \leqs \alpha(G) \leqs \beta(G) \leqs 3$.
\end{itemize}
\end{thm}

\begin{proof}
First assume $n$ is divisible by an odd prime $k$ with $n/k \geqs 4$. Then $G$ has a maximal $\C_3$-subgroup $H$ of type $O_{n/k}^{\e}(q^k)$ and Proposition \ref{p:bgs} gives $b(G,H) = 2$. Therefore, we may assume that $n = 2^m$ or $2k$, where $m \geqs 3$ and $k \geqs 5$ is a prime. 

If $n = 2^m$ with $m \geqs 3$ then $G$ has a maximal $\C_3$-subgroup $H$ of type $O_{n/2}^{\e}(q^2)$ and $b(G,H) \leqs 3$ by Proposition \ref{p:bgs}. The special case $G = \O_{8}^{+}(2)$ in part (ii) of the theorem can be checked directly (here we find  that $|H|^2>|G|$ for all $H \in \mathcal{M}$). 

Finally, let us assume $n=2k$, where $k \geqs 5$ is a prime. If $\e=-$ then we can take a $\C_3$-subgroup $H$ of type ${\rm GU}_{n/2}(q)$, in which case Proposition \ref{p:bgs} gives $b(G,H) \leqs 3$. Now assume $\e=+$. If $q$ is odd then $G$ has a maximal $\C_3$-subgroup $H$ of type $O_{n/2}(q^2)$ and the bound $b(G,H) \leqs 3$ follows from Proposition \ref{p:bgs}. Now assume $q$ is even. Let $\widetilde{V}$ be the natural module for $G$ and let $\O$ be the set of nondegenerate plus-type subspaces of dimension $k+1$. Then $G$ acts primitively on $\O$ and we claim that there is a base of size $3$.

To see this, write $k=2m+1$ and $\widetilde{V} = V \perp \la \widetilde{e},\widetilde{f} \ra$, where $V$ is a nondegenerate $4m$-space of plus-type and $Q(\widetilde{e})=Q(\widetilde{f})=0$ and $(\widetilde{e},\widetilde{f})=1$. Fix a standard basis for $V$ as in \eqref{e:B} and define the subspaces $W_0=W$, $W_1$ and $W_2$ of $V$ as in Lemma \ref{even}. Set  
\[
\widetilde{W}_0 = W \perp \langle \widetilde{e},\widetilde{f} \rangle, \;\; \widetilde{W}_1 = W_1 \perp \langle \widetilde{e},\widetilde{f} \rangle, \;\; \widetilde{W}_2 = W_2 \perp \langle e_1^{\ast}+\widetilde{e},e_2^{\ast}+\widetilde{f} \rangle
\]
and observe that $\widetilde{W}_i \in \O$ for $i=1,2,3$. 

Suppose $g \in G$ stabilises $\widetilde{W}_0$, $\widetilde{W}_1$ and $\widetilde{W}_2$. We claim that $g=1$. To see this, first observe that $g$ stabilises the nondegenerate $2$-space $\widetilde{W}_0 \cap \widetilde{W}_1 = \langle \widetilde{e},\widetilde{f} \rangle$, so $g$ also stabilises its orthogonal complement, namely $V$. Moreover $g$ stabilises $\widetilde{W}_i \cap V = W_i$, so Lemma \ref{even} implies that $g$ acts on $V$ as $T_a$ for some nonzero scalar $a \in \F$, and it acts on $\langle \widetilde{e},\widetilde{f} \rangle$ as a matrix $A \in O_{2}^{+}(q).$ There are two possibilities to consider.
\begin{itemize}\addtolength{\itemsep}{0.2\baselineskip}
\item[{\rm (a)}] Suppose $A=\begin{pmatrix}b&0\\0&b^{-1}
	\end{pmatrix}.$ Here 
	\[
	g(e_1^{\ast}+\widetilde{e})+g(e_2^{\ast}+\widetilde{f})-a(e_1^{\ast}+\widetilde{e})-a(e_2^{\ast}+\widetilde{f})=(b-a)\widetilde{e}+(b^{-1}-a)\widetilde{f} \in \widetilde{W}_2,
	\]
	 which implies that $b=a=b^{-1},$ and consequently $a=b=1$ and $g=1$.
		\item[{\rm (b)}] Otherwise, $A=\begin{pmatrix}0&b\\b^{-1}&0
		\end{pmatrix}.$ 	In this case, 
		\[
		g(e_1^{\ast}+\widetilde{e})+g(e_2^{\ast}+\widetilde{f})-a(e_1^{\ast}+\widetilde{e})-a(e_2^{\ast}+\widetilde{f})=(b^{-1}-a)\widetilde{e}+(b-a)\widetilde{f} \in \widetilde{W}_2
		\]
		and we deduce that $b=a=b^{-1}$ and thus $a=b=1.$ But then
		\[
		g=\begin{pmatrix}I_{4m}&0&0\\0&0&1\\0&1&0
		\end{pmatrix}\notin \Omega^+_{n}(q).
		\]
		\end{itemize}
	We conclude that $g=1$ as required.
\end{proof}

\vs

This completes the proof of Theorem \ref{t:main}.

\section{Proof of Corollary \ref{c:cor2}}\label{s:cor}

Suppose $G$ is a nonabelian finite simple group with ${\rm Maxdim}(G)={\rm Mindim}(G)$. As we observed in the introduction,
\[
{\rm Maxdim}(G)\geqs m(G)\geqs 3,
\] 
so Theorem \ref{t:main} implies that ${\rm Maxdim}(G)={\rm Mindim}(G)=m(G)=3.$

Since ${\rm Maxdim}(A_n)=n-2$ and ${\rm Mindim}(A_5)=2$, it follows that $G$ is not an alternating group. If $G$ is sporadic, then the condition ${\rm Mindim}(G)=3$ implies that $G={\rm M}_{22}$. However ${\rm M}_{22}$ has a maximal subgroup $H = \rm L_3(4)$ with $b({\rm M}_{22},H)=5$ (see \cite[Table 1]{bow}) and we deduce that ${\rm Maxdim}({\rm M}_{22})\geqs 5.$ If $G$ is an exceptional group of Lie type, then Theorem \ref{t:main} implies that $G = G_2(2)' \cong {\rm U}_{3}(3).$ By a theorem of Wagner \cite{Wag}, $G$ can be generated by $4$ involutions and no fewer, so $m(G) \geqs 4$ and thus ${\rm Maxdim}(G) \geqs 4$.

Finally, let us assume $G$ is a classical group with natural module of dimension $n$. As noted in the introduction to \cite{gs}, one needs at least $n$ conjugates of a fixed pseudoreflection to generate $G$ (here a pseudoreflection is an element whose fixed space is a hyperplane). Therefore, ${\rm Maxdim}(G)\geqs m(G)\geqs n$ and by applying 
Theorems \ref{t:lin} and \ref{t:uni} it follows that $G$ is one of ${\rm L}_2(7)$, ${\rm L}_2(9)$ and ${\rm U}_{3}(5).$ The first two possibilities can be ruled out since $m({\rm L}_2(7))=m({\rm L}_2(9))=4$ (see \cite{ws}). Finally, $G = {\rm U}_3(5)$ has a maximal subgroup $H= A_7$ and it is easy to check that $b(G,H) = 4$ (note that $|H|^3> |G|^2$).

\vs

This completes the proof of Corollary \ref{c:cor2}.

\section{A soluble example}\label{examples}

Recall that if $G$ is a nonabelian finite simple group, then $\a(G) - {\rm Mindim}(G) \leqs 1$. In stark contrast, in this final section we construct a family of finite soluble groups $G$ with the property that $\a(G) - {\rm Mindim}(G)$ can be arbitrarily large.

Let $F$ be the free group of rank $2$ and let $X$ be the intersection of the normal subgroups $N$ of $F$ with $F/N\cong D_8.$ It turns out that $X$ is a $2$-generated group of order $32$. A concrete construction of $X$ can be given in the following way. Let $D_8=\langle a,b \mid a^4,b^2,abab\rangle$ and consider the subgroup $X$ of $D_8\times D_8\times D_8$ generated by $x_1=(a,b,b)$ and $x_2=(b,ab,a).$ Note that the Frattini subgroup of $X$ is generated by the elements 
\[
y_1=x_1^2=(a^2,1,1), \;\; y_2=x_2^2=(1,1,a^2), \;\; y_3=[x_1,x_2]=(a^2,a^2,a^2).
\]
Also notice that $N_1=\langle y_1,y_2\rangle$, $N_2=\langle y_1,y_1y_2y_3\rangle$ and
$N_3=\langle y_2,y_1y_2y_3\rangle$ are normal subgroups of $X$ contained in $\frat(X),$ with $X/N_i\cong D_8$ for all $i$. The dihedral group $D_8$ can be viewed as an irreducible subgroup of ${\rm GL}_{2}(3),$ so there exists three irreducible $X$-modules $A_1,A_2,A_3$ with $|A_i|=9$ and $C_X(A_i)=N_i$ for all $i$. 

Consider the semidirect product $G=(A_1\times A_2\times A_3){:}X$. Clearly $\frat(G)=1$. The maximal subgroups of $G$ are divided into four families:
\begin{itemize}\addtolength{\itemsep}{0.2\baselineskip}
	\item $\mathcal M_0,$ the maximal subgroups of $G$ containing $A_1\times A_2\times A_3;$ 
	\item $\mathcal M_i,$ the maximal subgroups of $G$ supplementing $A_i,$ for $i=1,2,3$.
\end{itemize}

\vs

\noindent \emph{Claim.} $\a(G) \geqs 6$. 

\vs

To see this, let $\mathcal{A}$ be a set of maximal subgroups of $G$ with $\bigcap_{M\in \mathcal{A}}M=1.$ Now $\mathcal{A} \cap \mathcal M_i$ must be nonempty for $i=1,2,3$, say $\mathcal A$ contains $M_i\in \mathcal M_i$. 

We claim that there exists $a_i\in A_i$ such that $M_i=\left(\prod_{j\neq i}A_j\right)X^{a_i}$. We will do this for $i=1$, the other cases being similar. Write $G = A{:}X$ with $A = A_1A_2A_3$ and let $M$ be a maximal subgroup of $G$ with $MA_1=G$. Since $A_1$ is an abelian minimal normal subgroup of $G$, it follows that $M$ is a complement of $A_1$. Moreover, $A_1(M \cap A) = A_1M \cap A = G \cap A = A$ and $M \cap A \normeq G$, since $M \cap A$ is normal in $M$, and also in $A$ (since $A$ is abelian). Next observe that $[A_i,N_j]=A_i$ if $i \neq j$, which  implies that $A_2A_3=[A,N_1]=[A_1(M\cap A),N_1]=[M\cap A, N_1] \leqs M.$ Now the quotient $G/A_2A_3$ is a primitive solvable group of the form $(A/A_2A_3){:}(M/A_2A_3)$ and it also equals $(A/A_2A_3){:}(A_2A_3X/A_2A_3)$. Since any two complements of the socle of a primitive solvable group are conjugate, it follows that  there exists $a_1 \in A_1$ such that $M^{a_1^{-1}} = A_2A_3X$. Therefore, $M=A_2A_3X^{a_1}$. This justifies the claim.

Set $Y=M_1\cap M_2\cap M_3$ and observe that $Y=X^{a_1a_2a_3}.$ For $i=1,2,3$, let $Z_i=\bigcap_{M\notin \mathcal{M}_i}M$ be the intersection of all the maximal subgroups of $G$ which do not belong to $\mathcal{M}_i$, and let $R_i=\bigcap_{j\neq i}N_j.$ Notice that $Z_i=A_iR_i,$ so $Z_i\cap M_i=R_i^{a_i}\neq 1.$ Therefore, we must have $|\mathcal{A} \cap \mathcal M_i|\geqs 2$ for $i=1,2,3$ and consequently $\a(G)\geqs 6$ as required.

\vs 

Next we claim that ${\rm Mindim}(G) \leqs 5$. Set $B=A_1\times A_2\times A_3$ and $B_i=\prod_{j\neq i}A_j$ and consider the set $\mathcal{M}=\{M_1,M_2,M_3,M_4,M_5\}$ of maximal subgroups of $G$, where 
\[
M_1=B_1X, \; M_2=B_2X, \; M_3=B_3X,\; M_4=BK_1, \; M_5=BK_2
\]
and $K_1, K_2$ are two different maximal subgroups of $X.$ We claim that $\mathcal{M}$ is a maximal irredundant set of maximal subgroups of $G.$ It is clearly irredundant, so let us focus on maximality. Let $H$ be a maximal subgroup of $G$ with $H\notin \mathcal{M}.$ If $H \in \mathcal{M}_0$ then $H\cap M_4\cap M_5=M_4\cap M_5=B\frat(X),$ so $\mathcal{M} \cup \{H\}$ is redundant. If $H \in\mathcal{M}_i$ with $i>0$ then $M_1\cap M_2\cap M_3\cap H=C_X(a)$ for some $a \in A_i$, and one of $C_X(a) \cap K_1$ and $C_X(a) \cap K_2$ is contained in $\frat(X)$. So once again $\mathcal{M} \cup \{H\}$ is redundant. This proves that ${\rm Mindim}(G)\leqs 5$.

Finally, let $k$ be a positive integer and consider the direct product 
$\Gamma_k=G_1\times \cdots\times G_k$ with $G_i\cong G$ for each $i$. By repeating the above argument, it can be easily seen that $\alpha(\Gamma_k)\geqs 6k$ and ${\rm Mindim}(\Gamma_k)\leqs 5k,$ so 
\[
\alpha(\Gamma_k)-{\rm Mindim}(\Gamma_k)\geqs k 
\]
can be arbitrarily large.


\appendix

\section{On a primitive action of $G_2$ \\
by Timothy C. Burness and Robert M. Guralnick}\label{appendix}

Let $\bar{G}=G_2(k)$, where $k$ is an algebraically closed field of characteristic $p=2$, and let 
$\s$ be a Frobenius morphism of $\bar{G}$ such that $\bar{G}_{\s} = G_2(2)$ and $\bar{G}_{\s^e} = G = G_2(q)$ for some positive integer $e \geqs 2$ (so $q=2^e$). Let $\bar{H}=A_1\tilde{A}_1$ be a $\s$-stable subgroup of $\bar{G}$, where the second $A_1$ factor is generated by short root elements, and set 
$H = \bar{H}_{\s^e}$. Up to conjugacy, we may assume that 
\[
H = \la X_{3a+2b}, X_{-(3a+2b)}, X_{a}, X_{-a} \ra  = {\rm L}_2(q) \times {\rm L}_{2}(q)
\]
where $a$ and $b$ are simple roots for $G$, with $a$ short, $b$ long, and $X_{r}  = \{ x_{r}(t)\,:\, t \in \mathbb{F}_q\}$ is the root subgroup corresponding to the root $r$.

Our first result settles the case $q$ even in Lemma \ref{l:g2b}.

\begin{thm}\label{t:g2main}
If $g = x_{b}(1)x_{a+b}(1)x_{-b}(1) \in G$, then $H \cap H^g = 1$ and thus $b(G,H) = 2$.
\end{thm}

\begin{proof}
Set $L = H \cap H^g$ and note that $g \in \bar{G}_{\s}$. Suppose the result is false and choose $e$ minimal so that $L \ne 1$. With the aid of {\sc Magma} \cite{magma}, one checks that $L=1$ when $q=4$, so we have $q = 2^e$ with $e \geqs 3$. Since $e$ is minimal, it follows that $\s$ acts semiregularly on $L$, so $|L| \equiv 1 \imod{e}$. In particular,  some power of $\s$ acts on $L$ as a fixed point free automorphism of prime order and thus $L$ is nilpotent. 

It will be convenient to write $\bar{H} = \bar{H}_1\bar{H}_2$, where $\bar{H}_1$ and $\bar{H}_2$ are the $A_1$ factors of $\bar{H}$, with $\bar{H}_2$ generated by the short root subgroups $X_{a}$ and $X_{-a}$. Observe that the $\s$-stable subgroup $\bar{H} \cap \bar{H}^g$ is finite. Indeed, if it is infinite, then \cite[Proposition 8.1]{GG} would imply that $H \cap H^g$ is nontrivial when $q=4$, which is not the case. 

First we reduce to the case where $L$ is a $2$-group. Suppose $L$ is not a $2$-group and let $x \in L$ be a semisimple element of order $r \geqs 3$. Suppose $r \geqs 5$, or $r=3$ and $C_{\bar{G}}(x) \ne A_2$. Then $Z(C_{\bar{G}}(x))=T$ is a positive dimensional torus and $C_{\bar{G}}(x) = C_{\bar{G}}(T)$. Let $S$ be a maximal torus of $\bar{H}$ containing $x$. Then $T,S \leqs C_{\bar{G}}(T)$ and thus $T^y \leqs S$ for some $y \in C_{\bar{G}}(T)$ (since $T$ is contained in a maximal torus of $C_{\bar{G}}(T)$, and all such maximal tori are conjugate). Therefore, $T \leqs S < \bar{H}$. Similarly, since $x \in \bar{H}^g$ we deduce that $T < \bar{H}^g$ and thus $T \leqs \bar{H} \cap \bar{H}^g$. But this is not possible since $\bar{H} \cap \bar{H}^g$ is finite.

Now assume each nontrivial semisimple element $x \in L$ has order $3$ with $C_{\bar{G}}(x) = A_2$. By considering the restriction 
\[
\mathcal{L}(\bar{G}) \downarrow \bar{H} = \mathcal{L}(\bar{H}) \oplus (M_1 \otimes S^3(M_2)),
\]
where $\mathcal{L}(X)$ denotes the Lie algebra of $X$ and $M_i$ is the natural module for $\bar{H}_i$ (see \cite[Chapter 12]{Thomas}, for example), it is easy to see that $x \in \bar{H}_2$. Let $P$ be the unique Sylow $3$-subgroup of $L$. Then $P$ is contained in a Sylow $3$-subgroup of the second ${\rm L}_{2}(q)$ factor of $H$, which is cyclic, so $|P|=3$ and thus $P \leqs \bar{G}_{\s^2}$. It follows that
\[
P \leqs (\bar{H} \cap \bar{H}^g)_{\s^2},
\]
but this contradicts the fact that $H \cap H^g = 1$ when $q=4$. We have now reduced to the case where $L$ is a $2$-group. Note that $e \geqs 3$ is odd and $|L| \geqs 4$. 

Let $V$ be the natural $6$-dimensional irreducible module for $\bar{G}$ and recall that $\bar{G}$ preserves a symplectic form on $V$, so we can view $\bar{G}$ as a subgroup of ${\rm Sp}_{6}(K)$. In this setting, $\bar{H}$ is the stabiliser in $\bar{G}$ of a $2$-dimensional nondegenerate subspace $W$ of $V$ and one checks that $\la W, W^g\ra$ is a nondegenerate $4$-space (it suffices to work over $\mathbb{F}_2$ to verify this). It follows that $L$ fixes an orthogonal decomposition 
\begin{equation}\label{e:decc}
V = W \perp W' \perp W''
\end{equation}
of $V$ into $2$-dimensional nondegenerate subspaces. Set $M = W^{\perp}$ and note that $\bar{H}$ acts irreducibly on $M$, whence $M = M_1 \otimes M_2$, where $M_i$ is the natural module for $\bar{H}_i$.  In particular, $\bar{H}$ acts as ${\rm SO}_4(k)$ on $M$. Therefore, the stabiliser in $\bar{H}$ of both $W'$ and $W''$ is of the form $T.2$,  where $T$ is a maximal torus. But since $L < \bar{H}$ stabilises both subspaces, it follows that 
$|L| = 2$. This final contradiction completes the proof.   
\end{proof}

Let us consider the action of $\bar{G}$ on $\bar{\Omega} = \bar{G}/\bar{H}$ and define the base measures 
\[
b(\bar{G},\bar{H}),\; b^0(\bar{G},\bar{H}),\; b^1(\bar{G},\bar{H})
\]
as in \cite{BGS2}. Here $b(\bar{G},\bar{H})$ is the \emph{exact base size} of $\bar{G}$, which is the smallest integer $c$ such that $\bar{\Omega}$ contains $c$ points with trivial pointwise stabiliser. Similarly, the \emph{connected base size}, denoted $b^0(\bar{G},\bar{H})$, is the smallest $c$ such that $\bar{\Omega}$ contains $c$ points with finite pointwise stabiliser, and the \emph{generic base size} $b^1(\bar{G},\bar{H})$ is the minimal $c$ such that the product variety $\bar{\Omega}^c$ contains a nonempty open subvariety $U$ with the property that every $c$-tuple in $U$ is a base for $\bar{G}$. Evidently,  
\[
b^0(\bar{G},\bar{H}) \leqs b(\bar{G},\bar{H}) \leqs b^1(\bar{G},\bar{H}).
\]

By \cite[Lemma 3.21]{BGS2} we have $b^0(\bar{G},\bar{H}) = 2$ and $b^1(\bar{G},\bar{H}) \leqs 3$, but $b(\bar{G},\bar{H})$ and $b^1(\bar{G},\bar{H})$ were not determined precisely in \cite{BGS2}. The following theorem resolves this ambiguity. Since the statement only involves algebraic groups, we will choose to suppress the bar notation used above.

\begin{thm}\label{t:g2main2-alt}
Let $G = G_2(k)$ defined over an algebraically closed field $k$ of characteristic $2$ and let 
$H$ be a maximal rank subgroup of type $A_1\tilde{A}_1$.  Consider the natural action of $G$ on the quotient variety $\Omega = G/H$.
\begin{itemize}\addtolength{\itemsep}{0.2\baselineskip}
\item[{\rm (i)}] There exists a nonempty open subvariety $U \subseteq \Omega \times \Omega$ such that $G_{\a} \cap G_{\b}$ has order $2$ and contains a short root element  for all $(\a,\b) \in U$. 
\item[{\rm (ii)}] We have $b(G,H)=b^0(G,H)=2$ and $b^1(G,H)=3$.
\end{itemize}
\end{thm}

\begin{proof} 
Without loss of generality, we may assume that $k$ is the algebraic closure of the field of two elements. Let $H$ be the subgroup of $G$ generated by the root subgroups corresponding to the roots $\pm a$, $\pm (3a+2b)$ and let $g \in G_2(2)$ be the element defined in the statement of Theorem \ref{t:g2main}. Set $J = H \cap H^g$ and let $J(q)$ be the set of $\mathbb{F}_q$-points in $J$. By Theorem \ref{t:g2main} we have $J(2^e) = 1$ for all positive integers $e$ and thus $J=1$. Therefore, $b(G,H)=2$ and we deduce that $G_{\gamma} \cap G_{\delta}$ is finite on an open subvariety of $\Omega \times \Omega$, whence $b^0(G,H)=2$.
If (i) holds then $b^1(G,H) > 2$ and by \cite[Proposition 2.5(iv)]{BGS2} we have $b^1(G,H) \leqs b^0(G,H) + 1$, whence $b^1(G,H)=3$. Thus, part (ii) follows once we have proved (i).   
 
Let $V$ denote the natural $6$-dimensional module for $G$ and recall that $H$ is the  
stabiliser of a nondegenerate $2$-space (with respect to a $G$-invariant
symplectic form on $V$). Since $G$ acts transitively on such spaces, we can identify $\Omega$ with the set of nondegenerate $2$-dimensional subspaces
of $V$. Write $H = G_{\a}$ and let $X$ be the $2$-space corresponding to $\a$ under this identification.

Fix a diagonal involution $x \in H$ and note that $x$ is a short root element of $G$. Since $x^G \cap H$ is a union of two $H$-classes (those in a short root subgroup of $H$ and the diagonal involutions),
it follows that $\Omega(x)$, the set of fixed points of $x$ on $\Omega$, is a union of two $C_G(x)$ orbits.  More precisely, the two orbits are $C_G(x)\a$ and $C_G(x) g\a$, where $g \in G$ is such that $y=x^g$ is contained
in a short root subgroup $R$ of $H$.  Note that $C_G(x)$ is a $6$-dimensional irreducible variety, whence the two $C_G(x)$ orbits are irreducible varieties of dimensions $4$ and $2$, respectively.  Let $\Omega_0(x)$ denote the $4$-dimensional orbit.  

Let $\b \in C_G(x) g\a$ and let $Y$ be the $2$-space corresponding to $\b$. Since $C_G(y)=C_G(R)$, it follows that $G_{\a} \cap G_{\b}$ contains $R$ and thus $\langle X, Y \rangle$ cannot be a $4$-dimensional nondegenerate space (for then $R$ would stabilise this space and its orthogonal complement, as well $X$ and $Y$, whence $R$ would act quadratically on $V$, which it does not). On the other hand, there is clearly a $2$-dimensional nondegenerate $x$-invariant space $W$ such that $\langle X, W \rangle$ is a nondegenerate $4$-space. It follows that $W$ must correspond to an element in $\Omega_0(x)$, and since the nondegeneracy of $\langle X, W \rangle$ is an open condition, we conclude that this is true for a nonempty open subvariety of $\Omega_0(x)$.   

Next we claim that $H \cap G_{\gamma} = \langle x \rangle$ for a generic $\gamma \in \Omega_0(x)$ (that is, for all $\gamma$ in a nonempty open subvariety of $\Omega_0(x)$). Let $W$ be the $2$-space corresponding to $\gamma$. Then
$\langle X, W \rangle$ is nondegenerate and thus $H \cap G_{\gamma}$ acts quadratically on $V$. More precisely, we can write 
\begin{equation}\label{e:vv}
V = X \perp V_2 \perp V_3,
\end{equation}
where the summands are nondegenerate $2$-spaces with $V_2 \subseteq \langle X, W \rangle$ and $V_3 \subseteq \langle X, W \rangle^{\perp}$.

Since $H$ acts on $X^{\perp}$ as ${\rm SO}_4(k)$, any subgroup of $H$ stabilising a decomposition as in \eqref{e:vv} is contained in $\langle T, x \rangle$, where $T$ is a maximal torus of $H$.  Therefore, to justify the claim, it suffices to show
that $H \cap G_{\gamma}$ contains no semisimple elements. If this intersection contains a semisimple  element of order $r>3$, or an element of order $3$ whose centraliser is not ${\rm SL}_3(k)$, then the argument in the proof of Theorem \ref{t:g2main} shows that $H \cap G_{\gamma}$ contains a torus $S$ (namely, the centre of the centraliser of such a semisimple element). However, the set of fixed points of $S$ on $\Omega$ is at most $2$-dimensional (since $C_H(S)$ has codimension at most $2$ in $C_G(S)$). So for a generic $\gamma \in \Omega_0(x)$, the intersection $H \cap G_{\gamma}$ is either $\la x \ra$ as claimed, or it is isomorphic to the symmetric group $S_3$
(note that any elementary abelian subgroup of order $9$ in $H$ contains elements of order $3$ with centraliser not equal to ${\rm SL}_3(k)$). The centraliser in $G$ of such an $S_3$ subgroup is $3$-dimensional (indeed, the centraliser
of the element of order $3$ is ${\rm SL}_3(k)$ and the involution $x$ induces a graph automorphism on this subgroup). Since there are only finitely many $H$-classes of $S_3$ subgroups, it follows that the set of fixed points of $S_3$ on $\Omega$ is at most $3$-dimensional and this completes the proof of the claim.  

To complete the proof of the theorem, let us consider the morphism of varieties
\[
f: G \times \Omega_0(x) \times \Omega_0(x) \rightarrow \Omega \times \Omega
\]
given by $f(g,\b,\gamma)=(g\b,g\gamma)$. Consider the fiber $f^{-1}(\b,\gamma)$, where $(\b,\gamma) \in \Omega_0(x) \times \Omega_0(x)$. For a generic pair $(\b,\gamma)$, the previous claim implies that $G_{\b} \cap G_{\gamma} = \langle x \rangle$ and so if $(g,\delta,\epsilon) \in f^{-1}(\b,\gamma)$ then $g \in C_G(x)$.  Therefore, the dimension of the fiber coincides with the dimension of $C_G(x)$, which is $6$. In particular, the minimal dimension of a fiber
of $f$ is at most $6$ and thus the dimension of the image of $f$ is at least 
\[
14 + 4 + 4 - 6=16 = \dim (\Omega \times \Omega).
\]
Therefore, $f$ is dominant and for $(\delta,\epsilon)=(g\b,g\gamma) \in \Omega \times \Omega$ we have $G_{\delta} \cap G_{\epsilon} =\langle x^g \rangle$. The result follows. 
\end{proof} 

By combining Theorem \ref{t:g2main2-alt} with \cite[Theorem 3.13]{BGS2}, we get the following corollary.

\begin{cor}\label{c:g2_cor}
Let $G = G_2(k)$ defined over an algebraically closed field $k$ and let 
$H$ be a maximal rank subgroup of type $A_1\tilde{A}_1$.  Consider the natural action of $G$ on the quotient variety $\Omega = G/H$. Then $b(G,H)=b^0(G,H)=2$ and $b^1(G,H)=3$.
\end{cor}

\end{document}